\documentclass{amsart}
\usepackage{amsmath,amsthm}
\usepackage{amsfonts,amssymb}

\usepackage{color}

\usepackage{enumerate}

\newtheorem{thm}{Theorem}[section]
\newtheorem{cor}[thm]{Corollary}
\newtheorem{lem}[thm]{Lemma}
\newtheorem{prop}[thm]{Proposition}
\newtheorem{exam}[thm]{Example}

\newtheorem{defn}[thm]{Definition}

\theoremstyle{remark}
\newtheorem{rem}[thm]{Remark}

\def\bnu{{\boldsymbol{\large {\nu}}}} 
\def\bmu{{\boldsymbol{\large {\mu}}}} 
\def\bka{{\boldsymbol{\large {\kappa}}}} 
\def\bom{{\boldsymbol{\large {\omega}}}} 
\def\bal{{\boldsymbol{\large {\alpha}}}} 

 \def\a{{\alpha}}
 \def\b{{\beta}}
 \def\g{{\gamma}}
 \def\k{{\kappa}}
 \def\t{{\theta}}
 \def\l{{\lambda}}
 \def\d{{\delta}}
 
 \def\s{{\sigma}}
 \def\la{{\langle}}
 \def\ra{{\rangle}}
 \def\ve{{\varepsilon}}

 \def\hs{{\mathsf h}}
 \def\ks{{\mathsf k}}
 \def\sB{{\mathsf B}}
 \def\sC{{\mathsf C}}
 
 \def\sH{{\mathsf H}}
 \def\sK{{\mathsf K}}

 \def\sQ{{\mathsf Q}}
 \def\sR{{\mathsf R}}

 \def\xb{{\mathbf x}}
 \def\yb{{\mathbf y}}

 \def\CH{{\mathcal H}}

 \def\CV{{\mathcal V}}
 \def\BB{{\mathbb B}}

 \def\NN{{\mathbb N}}
 \def\PP{{\mathbb P}}
 \def\QQ{{\mathbb Q}}
 \def\RR{{\mathbb R}}
 \def\SS{{\mathbb S}}
 \def\ZZ{{\mathbb Z}}

\def\dim{\operatorname{dim}}

\newcommand{\wh}{\widehat}

\def\f{\frac}
\def\one{{\mathbf{1}}}       
\def\brho{{\boldsymbol{\rho}}}

\def\xt{\tilde x}
\def\bet{\tilde \beta}
\def\rt{\tilde \rho}
\def\nut{\tilde \nu}
\def\mut{\tilde \mu}
\newcommand{\fFt}[7]{{}_4F_3\left(\begin{matrix} #1 , #2, #3, #4 \\
#5, #6, #7 \end{matrix}\,; 1\right)}

\begin{document}

\title[Connection coefficients for orthogonal polynomials]
{Connection coefficients for classical orthogonal polynomials of several variables}

\date{January 20, 2017}

\author{Plamen~Iliev}
\address{P.~Iliev, School of Mathematics, Georgia Institute of Technology, 
Atlanta, GA 30332--0160, USA}
\email{iliev@math.gatech.edu}
\thanks{The first author is partially supported by Simons Foundation Grant \#280940.}

\author{Yuan~Xu}
\address{Y.~Xu, Department of Mathematics, University of Oregon, Eugene, 
OR 97403--1222, USA}
\email{yuan@uoregon.edu}
\thanks{The second author is partially supported by NSF grant \#1510296}


\begin{abstract}
Connection coefficients between different orthonormal bases satisfy two discrete orthogonal relations themselves. For classical 
orthogonal polynomials whose weights are invariant under the action of the symmetric group, connection coefficients between a 
basis consisting of products of hypergeometric functions and another basis obtained from the first one by applying a 
permutation are studied. For the Jacobi polynomials on the simplex, it is shown that the connection coefficients can be 
expressed in terms of Tratnik's multivariable Racah polynomials and their weights. This gives, in particular, a new interpretation 
of the hidden duality between the variables and the degree indices of the Racah polynomials, which lies at the heart of their 
bispectral properties. These techniques also lead to explicit formulas for connection coefficients of Hahn and Krawtchouk polynomials 
of several variables, as well as for orthogonal polynomials on balls and spheres.
\end{abstract}
 \maketitle

\section{Introduction}
\setcounter{equation}{0}

For a positive measure $\rho$ defined on $\RR^d$ that satisfies some mild assumptions, the space 
$\CV_n^d$ of orthogonal polynomials of degree $n$ in $d$ variables with respect to the inner product
$$  
  \la f,g\ra = \int_{\RR^d} f(x) g(x) d\rho(x)
$$ 
has dimension $\dim \CV_n^d = \binom{n+d-1}{n}$, where $0\neq P \in \CV_n^d$ if $P$ is a polynomial of degree $n$ and
$\la P, Q\ra =0$ for all polynomials $Q$ of degree less than $n$. The elements of a basis $\{P_\nu: \nu \in \NN_0^d, |\nu| =n\}$ for 
$\CV_n^d$ may not be orthogonal among themselves.  A basis whose elements are orthogonal to each other is called orthogonal, that is,  $\la P_\nu, P_\mu\ra = 0$ for $\nu \ne \mu$, and it is called orthonormal if, in addition, 
$\la P_\nu, P_\nu\ra = 1$. 
 
If $\PP_n: = \{P_\nu: \nu \in \NN_0^d, |\nu| =n\}$ and $\QQ_n:=\{Q_\nu: \nu \in \NN_0^d, |\nu| =n\}$ are two 
orthogonal bases of $\CV_n^d$, then we can express one in terms of the other; for example, 
\begin{equation}\label{eq:QccP}
   Q_\nu = \sum_{|\mu|=n} c_{\nu,\mu} P_\mu, \qquad |\nu| = n, \quad c_{\nu,\mu} \in \RR.
\end{equation}
We call the coefficients $c_{\nu,\mu}$ {\it connection coefficients} of $\QQ_n$ in terms of $\PP_n$. 
Throughout the paper we adopt the convention of using $\wh P_\nu=P_\nu/||P_\nu||$ to denote the orthonormal polynomial 
when $P_\nu$ is an orthogonal polynomial, and we denote by $\wh c_{\nu,\mu}$ the {\it normalized} 
connection coefficients between two orthonormal bases. Thus, \eqref{eq:QccP} can be rewritten as 
\begin{equation*}
   \wh Q_\nu = \sum_{|\mu|=n} \wh c_{\nu,\mu} \wh P_\mu, \qquad |\nu| = n.
\end{equation*}
If both $\{P_\mu\}$ and $\{Q_\nu\}$ are orthonormal bases, the matrix $\mathbf{C}:=(\wh c_{\nu,\mu})$ is orthogonal 
(\cite[p.\ 67]{DX}) and, therefore, satisfies 
\begin{equation}\label{eq:cc-2op}
  \sum_{|\omega| = n} \wh c_{\nu,\omega} \wh c_{\mu,\omega} = \delta_{\nu,\mu} \quad \hbox{and}\quad 
   \sum_{|\omega| = n}\wh  c_{\omega,\nu}\wh c_{\omega,\mu} = \delta_{\nu,\mu}.
\end{equation} 
We are interested in identifying the connection coefficients for classical orthogonal 
polynomials of both continuous and discrete variables. Especially interesting is the case when $c_{\nu,\mu}$ 
can be expressed explicitly in terms of discrete orthogonal polynomials and their weights.

Our starting point is the Jacobi polynomials of two variables that are orthogonal with respect to the weight function 
$x^\a y^\b(1-x-y)^\g $ on the triangle $T^2 =\{(x,y): x\ge 0, y \ge 0,1-x-y\ge0\}$. For this weight function, one 
orthonormal basis of $\CV_n^2$ consists of polynomials $P_j(x,y):= P_{n-j,j}^{\a,\b,\g}(x,y)$ for $0\le j \le n$ given 
in terms of the classical Jacobi polynomials. Another orthonormal basis consists of 
$Q_j(x,y):=P_{n-j,j}^{\b,\g,\a}(y, 1-x-y)$, obtained from $P_j$ by applying the permutation $(123)$ on $(\a,\b,\g)$ 
and $(x,y,1-x-y)$ simultaneously. It was shown in \cite{Du2} that the connection coefficients in this case are 
expressed in terms of the Racah polynomials (see Section \ref{se3})  by taking a limit of the 
corresponding results for Hahn polynomials, whereas the proof of the latter result in \cite{Du1} uses 
specific irreducible representations of the symmetric group.

We start with an elementary proof of the result in \cite{Du2} for two variables and then we consider the connection 
coefficients for the Jacobi polynomials of arbitrary number of variables on the simplex for bases generated by elements 
of the symmetric group. The picture in several variables quickly becomes much more complicated. The constructions 
require new ingredients and the results lead to interesting phenomena. For instance, in the case of three variables, the 
connection coefficients can be expressed by the Racah polynomials, the Racah polynomials of two variables defined
in \cite{Tr2}, or by the sum of products of such polynomials, depending on what permutation is applied in obtaining 
the second basis. Moreover, in the case of the Racah polynomials of two variables, the dual orthogonality provides 
a natural interpretation of a subtle duality between the degree indices and the variables, which lies at the heart of the 
bispectral property studied in \cite{GI}. Indeed, in our framework, the duality amounts to transposing the orthogonal 
matrix $\mathbf{C}$ and comparing the orthogonality relations in \eqref{eq:cc-2op}. All this is trivial for the one variable 
Racah polynomials, for which the symmetry between the variable and the degree index follows immediately from the 
explicit hypergeometric representation.

From the results about the Jacobi polynomials on the simplex, connection coefficients for several other families 
of classical orthogonal polynomials can be derived. First of all, using the Jacobi polynomials on the simplex as 
generating functions of the Hahn polynomials \cite{KM,X14}, the connection coefficients for the Hahn 
polynomials of several variables and their permutations can be derived from those for the Jacobi polynomials
on the simplex. Secondly, taking an appropriate limit (\cite{IX07,X14}), we can derive explicit formulas for the 
connection coefficients of the Krawtchouk polynomials of several variables, which can be expressed by Krawtchouk 
polynomials with different parameters. Finally, it is known that the orthogonal polynomials on the unit ball and the 
unit sphere and those on the simplex are related, which can be used to derive connection coefficients for orthogonal 
bases on these domains. In particular, this leads to explicit formulas for the connection coefficients for orthogonal bases
given in Cartesian coordinates and in polar coordinates, respectively, on the unit ball. 

It should be pointed out that this work is connected to several other areas. The explicit formulas of the connection 
coefficients for the Hahn polynomials of two variables in \cite{Du1} and those for the Jacobi polynomials on the 
simplex in \cite{Du2} were re-proved in \cite{KVdJ} from Lie and quantum algebra representations; see also \cite{LVdJ} 
and the references therein along this line. Classical multivariable orthogonal polynomials have also appeared recently 
in a large body of work around the quantum singular oscillator model and, more generally, quantum systems, which 
appear to share many common points with our work, see \cite{GV,GVZ} and reference therein. In the language 
of representation theory and mathematical physics, our connection coefficients share common trait with the 
Clebsch-Gordan coefficients and the $3n j$ symbols. While we adopt a more traditional special function approach, 
we believe that it would be interesting to find a Lie-theoretic interpretation of the new results concerning the 
multivariable Racah polynomials in the present paper. Note that for the large family of multivariable Krawtchouk polynomials 
introduced in \cite{G71}, such Lie interpretation together with their spectral properties were obtained in \cite{I}. For a 
recent introduction to these polynomials and numerous probabilistic applications, see \cite{DG} and the references 
therein.

The paper is organized as follows. In the next section we recall the Jacobi polynomials on the simplex, define the 
connection coefficients for these polynomials and their permutations, and study their general properties. Since the 
explicit formulas for these coefficients are given in terms of the Racah polynomials of one or several variables, we collect 
the results for the Racah polynomials in Section \ref{se3}, including several new results clarifying the dual Racah 
orthogonality. Explicit formulas of the connection coefficients in dimension 2 are given in Section \ref{se4}, which serves 
as a basis for further study, and Section \ref{se5} contains a detailed study of all connection coefficients in dimension 3. 
In Section \ref{se6}, we present the general results in arbitrary dimension and, in particular, we give the explicit 
formula for the cyclic permutation in terms of the Racah polynomials of several variables. The connection coefficients 
for the Hahn and Krawtchouk polynomials are discussed in Sections \ref{se7} and \ref{se8}, respectively. Finally, the 
connection coefficients for orthogonal polynomials on the unit ball and the unit sphere are studied in Section \ref{se9}.

\section{Connection coefficients for Jacobi polynomials on the simplex}\label{se2}
\setcounter{equation}{0}

The Jacobi weight function on the simplex $T^d: = \{x \in \RR^d: x_i \ge 0, |x|\le 1\}$, where $|x|:=x_1+\ldots + x_d$,
is defined by 
\begin{equation}\label{eq:weightW}
  W_\kappa (x):=  x_1^{\kappa_1} \cdots x_d^{\kappa_d} (1-|x|)^{\kappa_{d+1}},  \quad x \in T^d,
\end{equation}
where $\k \in \RR^{d+1}$ with $\k_1 > -1,\ldots, \k_{d+1} > -1$. The Jacobi polynomials of $d$ variables 
are orthogonal polynomials with respect to the inner product
\begin{equation}\label{eq:innerW}
   \la f, g \ra_{W_\k} : = \frac{\Gamma(|\kappa| + d+1)}
       {\prod_{i=1}^{d+1}\Gamma(\kappa_i +1)} \int_{T^d} f(x) g(x) W_\k(x) dx. 
\end{equation}
The normalizing factor in front of the integral is chosen so that $\la 1, 1 \ra_{W_\k} =1$.
Let $\CV_n^d(W_\k)$ denote the space of orthogonal polynomials of degree $n$ with respect to $W_\k$. 
A family of orthogonal polynomials with respect to $W_\k$ can be given explicitly in terms of the Jacobi
polynomials. For $\a, \b > -1$, the Jacobi polynomial $P_n^{(\a,\b)}$ satisfies
\begin{align}
P_n^{(\a,\b)}(x) = & \frac{(\a+1)_n}{n!} {}_2F_1 \left (\begin{matrix} -n, n+\a+\b+1 \\ \a+1 \end{matrix}; \frac{1-x}2 \right)
      \label{Jacobi1}\\
  = &  \frac{(\a+1)_n}{n!} \left(\frac{1+x}2 \right)^n{}_2F_1 \left (\begin{matrix} -n, -n-\b \\ \a+1 \end{matrix}; \frac{x-1}{x+1} \right).
      \label{Jacobi2}
\end{align}

To describe our orthogonal polynomials on the simplex, we will need the following notations that will be used 
throughout this paper: 
For $y=(y_1,\ldots, y_{d}) \in \RR^{d}$ and $1 \le j \le d$, we define 
\begin{equation}\label{xsupj}
    \yb_j := (y_1, \ldots, y_j) \quad \hbox{and}\quad \yb^j := (y_j, \ldots, y_d), 
\end{equation}
and also define $\yb_0 := 0$ and $\yb^{d+1} :=  0$. It follows that $\yb_d = \yb^1 = y$, 
and 
$$
   |\yb_j| = y_1 + \cdots + y_j,   \quad |\yb^j| = y_j + \cdots + y_d, \quad\hbox{and}\quad
   |\yb_0| = |\yb^{d+1}| = 0.
$$
For $\k = (\k_1,\ldots, \k_{d+1})$, we have $\bka^j := (\k_j, \ldots, \k_{d+1})$ for $1 \le j \le d+1$ and adopt
this notation also for other Greek letters. For $\nu \in \NN_0^d$ and $\k \in \RR^{d+1}$, we define
\begin{equation}\label{eq:aj}
   a_j:=a_j(\kappa,\nu):=|\bka^{j+1}| + 2 |\bnu^{j+1}| + d-j, \qquad 1 \le j \le d.
\end{equation} 
Notice that $a_{d} = \k_{d+1}$ since $|\bnu^{d+1}| =0$ by definition. The standard basis on the simplex
is given in the following proposition \cite[p. 150]{DX}.

\begin{prop}
For $\k \in \RR^{d+1}$ with $\k_i > -1$, $\nu \in \NN_0^d$ and $x \in \RR^d$, define
\begin{equation}\label{eq:Pnu}
  P_\nu^\kappa (x) := \prod_{j=1}^d \left(1-|\xb_{j-1}| \right)^{\nu_j} 
               P_{\nu_j}^{(a_j,\kappa_j)}\left (\frac{2x_j}{1-|\xb_{j-1}|} -1\right), 
\end{equation}
where $a_j = a_j(\k, \nu)$ is defined in \eqref{eq:aj}. The polynomials in $\{P_\nu^\k :|\nu|=n\}$ form an 
orthogonal basis of $\CV_n^d(W_\k)$ with 
$A_\nu(\k)  = \langle P_\nu^\k, P_\nu^\k \rangle_{W_\kappa}$ given by  
\begin{align} \label{eq:norm1}
  A_\nu(\k) := 
   \frac{1}{(|\kappa|+d+1)_{2|\nu|}}
     \prod_{j=1}^d \frac{(\k_j+a_j+1)_{2 \nu_j} (\k_j +1)_{\nu_j} (a_j+1)_{\nu_j}} {(\kappa_j+a_j+1)_{\nu_j}\nu_j!}.
\end{align}
\end{prop}
We shall call the polynomials $P_\nu^\k$ the Jacobi polynomials on the simplex. 

Let $S_{d+1}$ be the symmetric group consisting of all permutations of $d+1$ symbols. 
For $\tau\in S_{d+1}$ and $y\in\RR^{d+1}$ we define
$$\tau y:=(y_{\tau(1)},\dots,y_{\tau(d+1)}).$$
To capture the symmetry of the simplex, we make the following definition. 
\begin{defn}\label{ppp}
For $\tau\in S_{d+1}$ and $x\in T^d$, we define $\tau x$ as the first $d$ components of $\tau (x_1,\ldots, x_d, x_{d+1})$, where $x_{d+1}: =1-|x|$; i.e. $\tau x:=(x_{\tau(1)},\dots,x_{\tau(d)})\in T^{d}$.
\end{defn} 

Evidently the weight function $W_\k$ satisfies $W_{\tau\k}(\tau x) = W_\k(x)$. Furthermore, it is easy to see that the
set of polynomials $\{P_\nu^{\tau\k}(\tau x): |\nu| =n\}$ is also an orthogonal basis of $\CV_n^d(W_\k)$. 
Moreover, the square of the norm of $P_\nu^{\tau\k}(\tau\{\cdot\})$ is $A_\nu(\tau\k )$. Indeed, using the fact 
that $W_\k(\tau^{-1} y) = W_{\tau\k}(y)$, 
\begin{align*}
  \int_{T^d} \left[ P_\nu^{\tau\k}(\tau x)\right]^2 W_k(x)dx = \int_{T^d} \left[ P_\nu^{\tau\k}(y)\right]^2 W_{\tau\k}(y)dy
     = A_\nu(\tau\k). 
\end{align*}
Since $\{P_\nu^{\k}: |\nu| =n\}$ and $\{P_\nu^{\tau\k}(\tau\{\cdot\}): |\nu| =n\}$ are both bases of $\CV_n^d(W_\k)$, we 
can write one basis in terms of the other. 

\begin{defn}\label{permact}
Let $\tau \in S_{d+1}$ and $\nu, \mu \in \NN_0^d$ with $|\nu| = |\mu| =n$. The connection coefficients $ c_{\nu,\mu}^\tau(\k)$
are defined by 
\begin{equation} \label{defn:cc}
  P_\nu^{\tau\k}(\tau x) = \sum_{|\mu|= n} c_{\nu,\mu}^\tau(\k) P_\mu^\k (x). 
\end{equation}
\end{defn}

Let $\wh P_\nu^\k (x) =  A_\nu(\k)^{-1/2} P_\nu^\k (x)$. Then $\{ \wh P_\nu^\k: |\nu| =n\}$ is an orthonormal basis of 
$\CV_n^d$. We can write \eqref{defn:cc} as 
\begin{equation}\label{eq:conn-c-coeff}
    \wh P_\nu^{\tau\k}(\tau x) = \sum_{|\mu|= n} \wh c_{\nu,\mu}^{\, \tau}(\k) \wh P_\mu^\k (x), \quad\hbox{with}\quad
        \wh c_{\nu,\mu}^{\, \tau}(\k) := \sqrt{\frac{A_\mu(\k)}{A_\nu(\tau\k)} } c_{\nu,\mu}^\tau(\k),  
\end{equation}
where $\wh c_{\nu,\mu}^{\, \tau}(\k)$ are the normalized connection coefficients. 

The connection coefficients satisfy the following useful relation.
\begin{prop} \label{prop:convolu}
Let $\tau_1$ and $\tau_2$ be two elements of $S_{d+1}$. Then, for $\mu,\nu \in \NN_0^{d}$ and $|\nu| = |\mu| =n$, 
\begin{align} \label{eq:c-convolution}
   c_{\nu,\mu}^{\tau_1\tau_2}(\k) = \sum_{|\omega|=n} c_{\nu,\omega}^{\tau_2} (\tau_1\k) c_{\omega,\mu}^{\tau_1}(\k).
\end{align}
Moreover, formula \eqref{eq:c-convolution} holds also for the normalized connection coefficients.
\end{prop}
 
\begin{proof}
By \eqref{defn:cc}, we have  $P_\nu^{\tau_2\k}(\tau_2x)  = \sum_{|\omega| = n} c_{\nu, \omega}^{\tau_2} (\k) P_\omega^{\k}(x).$ Applying $\tau_1$ on both sides we get
\begin{align*}
  P_\nu^{\tau_1\tau_2\k}(\tau_1\tau_2x) & = \sum_{|\omega| = n} c_{\nu, \omega}^{\tau_2} (\tau_1\k) P_\omega^{\tau_1\k}(\tau_1x) \\
       & = \sum_{|\omega| = n} c_{\nu, \omega}^{\tau_2} (\tau_1\k)  \sum_{|\mu| = n} c_{\omega,\mu }^{\tau_1}(\k)
             P_\mu^{\k}(x) \\
       & = \sum_{|\mu| = n} \sum_{|\omega| = n} c_{\nu, \omega}^{\tau_2} (\tau_1\k) c_{\omega,\mu }^{\tau_1}(\k) 
             P_\mu^{\k}(x), 
\end{align*}
from which the identity \eqref{eq:c-convolution} follows from the definition of $c_{\nu,\mu}^{\tau_1\tau_2}(\k)$. The proof of \eqref{eq:c-convolution} for the normalized coefficients follows along the same lines by working with formula \eqref{eq:conn-c-coeff}.
\end{proof}

As an immediate corollary we obtain the following useful relations for the normalized connection coefficients.

\begin{prop}
For $\nu, \mu \in \NN_0^{d}$ with $|\nu| = |\mu| =n$ we have
\begin{align}\label{eq:discreteOP}
     \sum_{|\omega| =n}  \wh c_{\nu,\omega}^{\, \tau}(\k)  \wh c_{\mu,\omega}^{\, \tau}(\k) = \delta_{\nu,\mu},
        \quad\hbox{}\quad
     \sum_{|\omega| =n}  \wh c_{\omega,\nu}^{\, \tau}(\k)  \wh c_{\omega,\mu}^{\, \tau}(\k) = \delta_{\nu,\mu},
\end{align}
and
\begin{equation} \label{eq:inverse}
\wh c_{\nu,\mu}^{\, \tau^{-1}}(\k) = \wh c_{\mu,\nu}^{\, \tau}(\tau^{-1}\k).
\end{equation}
\end{prop}

Directly from the orthogonality, it follows that 
$$
   c_{\nu,\mu}^\tau(\k) = \frac{1}{A_{\nu}(\k)} \la  P_\mu^{\tau\k}(\tau\{\cdot\}),P_\nu^\k \ra_{W_\k}.
$$
We are interested in finding an explicit expression for $c_{\nu,\mu}^\tau(\k)$. It turns out that these coefficients can be regarded
as polynomials, which satisfy, by \eqref{eq:discreteOP}, a discrete orthogonality relation. In fact, these discrete orthogonal 
polynomials can be expressed by the Racah polynomials. 

\section{Racah polynomials}\label{se3}
\setcounter{equation}{0}

Since the explicit formula of the connection coefficient $c_{\nu,\mu}^\tau(\k)$ involves the Racah polynomials, 
we collect what we need for these polynomials in this section. 

The classical Racah polynomials in one-variable are defined by 
$$
   R_n(\l(x)):= R_n(\l(x); \a,\b,\g,\delta) 
    =  {}_4F_3 \left (\begin{matrix} -n, n+\a+\b+1,-x, x+\g+\delta +1\\ \a+1,\b+\delta+1,\g+1 \end{matrix}; 1 \right) 
$$ 
for $n=0,1,\ldots, N$ and $\l(x) = x(x+\g+\d+1)$ and one of the bottom parameters is $-N$, see \cite[p.~190]{KLS}. They satisfy the orthogonality relation
$$
    \sum_{x=0}^N w_R(x) R_n(\l(x)) R_m(\l(x)) = r_n \delta_{m,n},
$$ 
where $r_n$ is the square of the norm of $R_n$ and $w_R$ is the weight function given by
\begin{align} \label{eq:Racah-weight}
   w_R(x): = & \,w(x;\a,\b,\g,\delta) \\
            =  & \,   \frac{(\g+\delta+1)_x ((\g+\delta+3)/2)_x (\a+1)_x(\b+\delta+1)_x (\g +1)_x}
       {x!((\g+\delta+1)/2)_x (\g+\delta -\a+1)_x(\g -\b+1)_x (\delta +1)_x}. \notag
\end{align}

For $d \ge 2$, the Racah polynomials of $d$ variables are defined on the simplex 
$V_N=\{x \in \NN_0^d: 0 \le x_1 \le x_2 \le \cdots \le x_d \le  N\}$ and they are orthogonal with respect to the 
weight function
\begin{align}\label{eq:d-Racah-weight}
     w_\sR(x;\beta,N):= \prod_{j=0}^d \frac{(\b_{j+1}-\b_j)_{x_{j+1}-x_j}(\b_{j+1})_{x_{j+1}+x_j}}
       {(x_{j+1}-x_j)! (\b_{j}+1)_{x_{j+1}+x_j}} \prod_{j=1}^d \frac{((\b_j+2)/2)_{x_j}}{(\b_j/2)_{x_j}},
\end{align}
where $x_0=0$ and $x_{d+1}=N$.
A basis for $\CV_n(w_\sR)$ is given explicitly by $\{\sR_\nu: \nu \in \NN_0^d,\;|\nu| = n\}$, where 
\begin{align*}
  & \sR_\nu(x;\beta, N):=   \prod_{j=1}^d (2 |\bnu_{j-1}| +\b_j-\b_0)_{\nu_j} 
   ( |\bnu_{j-1}|+\b_{j+1}+x_{j+1})_{\nu_j}( |\bnu_{j-1}|-x_{j+1})_{\nu_j} \\
      & \qquad \quad \times  
       {}_4F_3 \left (\begin{matrix} -\nu_j, \nu_j+  2 |\bnu_{j-1}| +\b_{j+1}-\b_0-1,   |\bnu_{j-1}|-x_j,   |\bnu_{j-1}| +\b_j+x_j\\
      2 |\bnu_{j-1}| +\b_j-\b_0, |\bnu_{j-1}|+\b_{j+1}+x_{j+1},  |\bnu_{j-1}|-x_{j+1} \end{matrix}; 1 \right).
\end{align*}
These polynomials are mutually orthogonal with respect to $w_\sR$, that is, 
$$
  \sum_{x\in V_N} \sR_\nu(x;\beta,N) \sR_\mu(x;\beta,N) w_\sR(x;\beta,N) = \delta_{\nu,\mu}  \left[ r_\nu(\beta,N) \right]^2,
$$
where $r_\nu(\beta,N)$ is the norm of $\sR_\nu(\cdot;\beta,N)$. Using formula (2.4) in \cite{Tr2}, one can show that
\begin{align}\label{eq:norm}
   &  [r_\nu(\beta,N)]^2=\frac{(\b_{d+1})_{N+|\nu|} (-N)_{|\nu|}(-N-\b_0)_{|\nu|} (2|\nu|+\b_{d+1}-\b_{0})_{N-|\nu|}}{N!\,(\b_0+1)_N}\\ 
      \times 
     \prod_{k=1}^{d}& \nu_k!(\b_{k+1}-\b_{k})_{\nu_k}(2|\bnu_{k-1}|+\b_{k}-\b_{0})_{\nu_k}(|\bnu_{k}|+|\bnu_{k-1}|+\b_{k+1}-\b_{0}-1)_{\nu_k}
     \notag.
\end{align}

\begin{rem} 
When $d=1$, the formulas above define the same polynomials, up to some unessential factors. More precisely, if we set 
\begin{equation}\label{eq:connection}
\alpha=-N-1,\quad \beta=\beta_2-\beta_0-1+N,\quad \gamma=\beta_1-\beta_0-1,\quad \delta=\beta_0 ,
\end{equation}
it is not hard to see that weights in \eqref{eq:Racah-weight} and \eqref{eq:d-Racah-weight} differ by a factor independent of $x$ and the  ${}_4F_3$ functions defining the polynomials coincide. In particular, if we consider the orthonormal polynomials we have  $\hat{\sR}_n(x)=(-1)^n\hat{R}_n(\lambda(x))$.
\end{rem}

Next we formulate an important duality relation which lies at the heart of the bispectral property of the Racah polynomials studied in \cite{GI}. Let us define dual indices $\nut$, variables $\xt$, and parameters $\bet$ by
\begin{equation}
\begin{alignedat}{2}
  \xt_j&=N-|\bnu_{d+1-j}| && \text{   for } 
                                 j=1,\dots, d,\label{eq:dualvariables}\\
\nut_j&=x_{d+2-j}-x_{d+1-j} &&\text{   for } j=1,\dots, d, \\
\bet_0&=\beta_0, \\
\bet_j&=\beta_0-\beta_{d+2-j}-2N+1 &&\text{   for }j=1,\dots, d+1. 
\end{alignedat}
\end{equation}

\begin{prop}\label{prop:duality}
The map $(x,\b,\nu,N)\to(\xt,\bet,\nut,N)$ is an involution. Moreover, the Racah polynomials satisfy the following duality relation
\begin{align}\label{eq:duality}
 & \frac{\sR_{\nu}(x;\beta,N)}{(-N)_{|\nu|}(-N-\beta_0)_{|\nu|}\prod_{j=1}^{d}(\beta_{j+1}-\beta_{j})_{\nu_j}} \\
  & \qquad\qquad\qquad =\frac{\sR_{\nut}(\xt;\bet,N)}{(-N)_{|\nut|}(-N-\bet_0)_{|\nut|}\prod_{j=1}^{d}(\bet_{j+1}-\bet_{j})_{\nut_j}}. \notag
\end{align}
\end{prop}
\begin{proof}
The duality can be obtained from the one established in \cite{GI}, by applying the involutions $x_j\to-\b_j-x_j$ to the duality relations (4.1) and Theorem 4.4 there. Alternatively, we give a direct proof as follows. Applying the Whipple identity
\begin{align}
&(U)_m(V)_m(W)_m\;\fFt{-m}{X}{Y}{Z}{U}{V}{W}   =(1-V+Z-m)_m \label{Whipple}\\
& \quad \times (1-W+Z-m)_m(U)_m\;
\fFt{-m}{U-X}{U-Y}{Z}{1-V+Z-m}{1-W+Z-m}{U}\notag
\end{align}
with $m=\nu_j$, $X=\nu_j+2|\bnu_{j-1}|+\b_{j+1}-\b_0-1$, $Y=|\bnu_{j-1}|-x_j$, $Z=|\bnu_{j-1}|+\b_j+x_{j}$, $U=|\bnu_{j-1}|-x_{j+1}$,  
$V=2|\bnu_{j-1}|+\b_{j}-\b_0$,  $W=|\bnu_{j-1}|+\b_{j+1}+x_{j+1}$, we see that the $j$th term in the product defining 
the Racah polynomials can be rewritten as
\begin{align*}
&(|\bnu_{j-1}| -x_{j+1})_{\nu_j} 
   (1- |\bnu_{j}|+\b_0+x_{j})_{\nu_j}(1- \nu_{j}+\b_{j}-\b_{j+1}+x_{j}-x_{j+1})_{\nu_j} \\
  &\quad    \times  
       {}_4F_3 \left (\begin{matrix} -\nu_j, -x_{j+1}+x_{j}, -x_{j+1}-|\bnu_j| -\b_{j+1}+\b_0+1,  |\bnu_{j-1}| +\b_j+x_j\\
       |\bnu_{j-1}| -x_{j+1}, 1- |\bnu_{j}|+\b_0+x_{j}, 1- \nu_{j}+\b_{j}-\b_{j+1}+x_{j}-x_{j+1}\end{matrix}; 1 \right).
\end{align*}
Plugging the dual variables \eqref{eq:dualvariables} in the last formula, one can show that the ${}_4F_3$ term coincides with the ${}_4F_3$ term above with $j$ replaced by $d+1-j$. Thus all ${}_4F_3$ terms in $\sR_{\nu}(x;\beta,N)/\sR_{\nut}(\xt;\bet,N)$ cancel and after we simplify and rearrange the remaining products, we obtain equation \eqref{eq:duality}.
\end{proof}

Using formulas \eqref{eq:d-Racah-weight}, \eqref{eq:norm} and \eqref{eq:dualvariables} one can check that 
\begin{equation}
\frac{r_{\nu}(\b,N)^2\,  w_\sR(\xt;\bet,N)}{\left[(-N)_{|\nu|}(-N-\b_0)_{|\nu|} \prod_{k=1}^{d}(\b_{k+1}-\b_{k})_{\nu_k}\right]^2}
=\frac{(\b_{d+1})_{2N}(\bet_{d+1})_{2N}}{[N!\,(\b_0+1)_N]^2},
\end{equation}
which is a constant independent of $\nu$. 

Combining Proposition~\ref{prop:duality} with the last equation we obtain the following corollary.
\begin{cor}\label{cor:dualorthogonality}
The orthonormal Racah polynomials satisfy the duality relation 
\begin{equation}\label{eq:dualorthogonality}
\sqrt{w_\sR(x; \beta, N)}\, \wh \sR_{\nu}(x; \beta, N)=\sqrt{w_\sR(\xt; \bet, N)}\, \wh \sR_{\nut}(\xt; \bet, N).
\end{equation}
\end{cor}

\begin{rem} \label{rem:Racah2}
Following Tratnik \cite{Tr2}, let us define new variables, indices and parameters by 
\begin{equation}
\begin{alignedat}{2}
   x'_j&=N-x_{d+1-j} && \text{  for } 
                                 j=1,\dots, d,\label{eq:conjvariables}\\
\nu'_j&=\nu_{d+1-j} &&\text{  for } j=1,\dots, d,  \\
\b_j'&=-2N-\b_{d+1-j} &&\text{  for }j=0,1,\dots, d+1. 
\end{alignedat}
\end{equation}
Then it is easy to check that
\begin{equation}
\sR_{\nu}(x;\beta,N)=\sR'_{\nu'}(x';\beta',N),
\end{equation}
where $\sR'$ is defined by
\begin{align*}
  & \sR'_\nu(x;\beta, N):=   \prod_{j=1}^d (2 |\bnu^{j+1}| +\b_{d+1}-\b_j)_{\nu_j} \\
   & \qquad\qquad \qquad\qquad \times (|\bnu^{j+1}|-N-\b_{j-1}-x_{j-1})_{\nu_j}(|\bnu^{j+1}|-N+x_{j-1} )_{\nu_j} \\
      &   \times  
       {}_4F_3 \left (\begin{matrix} -\nu_j, \nu_j+  2 |\bnu^{j+1}| +\b_{d+1}-\b_{j-1}-1,   |\bnu^{j+1}|-N+x_j,   |\bnu^{j+1}| -N-\b_j-x_j\\
      2 |\bnu^{j+1}| +\b_{d+1}-\b_j, |\bnu^{j+1}|-N-\b_{j-1}-x_{j-1},  |\bnu^{j+1}|-N+x_{j-1} \end{matrix}; 1 \right).
\end{align*}
Since the Racah weights $w_{\sR}(x',\b',N)$ and $w_{\sR}(x,\b,N)$ differ by a factor independent of $x$, we see that the last formula defines a second family of Racah polynomials of $d$ variables. In particular, if we combine this with Corollary~\ref{cor:dualorthogonality}, we obtain
\begin{equation}\label{eq:dualorthogonality2}
\sqrt{w_\sR(x; \beta, N)}\, \wh \sR_{\nu}(x; \beta, N)=\sqrt{w_\sR(\xt'; \bet', N)}\, \wh \sR'_{\nut'}(\xt'; \bet', N),
\end{equation}
where 
\begin{equation}
\begin{alignedat}{2}
\xt_j'&=|\bnu_{j}| && \text{  for } 
                                 j=1,\dots, d,\label{eq:dualvariables2}\\
\nut_j'&=x_{j+1}-x_{j} &&\text{  for } j=1,\dots, d,  \\
\bet_j'&=\beta_{j+1}-\beta_{0}-1 &&\text{  for }j=0,\dots, d, \\
\bet_{d+1}'&=-2N-\beta_0. 
\end{alignedat}
\end{equation}
\end{rem} 

\section{Jacobi polynomials of two variables}\label{se4}
\setcounter{equation}{0} 

In this section we consider the Jacobi polynomials on the triangle, or simplex when $d=2$. In this case, 
the basis \eqref{eq:Pnu} of $\CV_n^2(W_\k)$ consists of $P_{\nu}$ for $\nu_1+\nu_2 = n$. 
For later use, however, we will write $\nu = (n-j,j)$, so that the polynomials in \eqref{eq:Pnu} become 
\begin{equation} \label{eq:Pjn}
   P_{n-j,j}^\k(x_1,x_2) = P_{n-j}^{(\k_2+\k_3+2j+1,\k_1)}(2x_1-1) (1-x_1)^j P_j^{(\k_3,\k_2)}\left (\frac{2 x_2}{1-x_1}-1\right)
\end{equation} 
for $0 \le j \le n$. We use the usual notation for permutations to denote the elements of the symmetric group $S_3$, 
$$
  S_3 = \{ (1), (12), (13), (23), (123), (132)\}.
$$
Let $\tau$ be an element of $S_3$ and $x=(x_1,x_2) \in \RR^2$. By Definition~\ref{ppp}, $\tau x = (x_{\tau(1)}, x_{\tau(2)})$, 
where $x_3 = 1-x_1-x_2$. To simplify the notation, we define
$$
  P_{n-j,j}^{\tau; \kappa} (x_1,x_2) :=  P_{n-j,j}^{\tau \kappa}(x_{\tau(1)}, x_{\tau(2)}).  
$$
Accordingly, we denote the connection coefficients by $c_{j,m}^\tau(\k, n)$ and  \eqref{defn:cc} becomes 
\begin{equation} \label{eq:con-coe-2d}
   P_{n-j,j}^{\tau;\k}(x_1,x_2) = \sum_{m=0}^n c_{j,m}^{\tau}(\k,n) P_{n-m,m}^{\k}(x_1,x_2). 
\end{equation}

Under permutations of the basis $P_{n-j,j}^\k$, there are essentially three distinct families of orthogonal polynomials. 
The first one consists of $P_{n-j,j}^\k$ given above, the second one consists of   
\begin{equation} \label{eq:Pjn(12)}
 P_{n-j,j}^{(12);\k} (x_1,x_2)  =  P_{n-j}^{(\k_1+\k_3+2j+1,\k_2)}(2x_2-1) (1-x_2)^j P_j^{(\k_3,\k_1)}\left (\frac{2 x_1}{1-x_2}-1\right)
\end{equation}
for $0 \le j \le n$, and the third one consists of
\begin{equation} \label{eq:Pjn(13)}
 P_{n-j,j}^{(13);\k} (x_1,x_2)  =  P_{n-j}^{(\k_1+\k_2+2j+1,\k_3)}(1-2x_1-2x_2) (x_1+x_2)^j P_j^{(\k_1,\k_2)}\left (\frac{2 x_2}{x_1+x_2}-1\right) 
\end{equation}
for $0 \le j \le n$. Indeed, there are essentially these three families since it is easy to see that 
\begin{align} \label{eq:trivial-2d}
    P_{n-j,j}^{(23);\k} (x_1,x_2) &= (-1)^j P_{n-j,j}^{\k} (x_1,x_2), \notag \\
    P_{n-j,j}^{(123);\k} (x_1,x_2) &= (-1)^j P_{n-j,j}^{(12);\k} (x_1,x_2), \\
    P_{n-j,j}^{(132);\k} (x_1,x_2) &= (-1)^j P_{n-j,j}^{(13);\k} (x_1,x_2).  \notag  
\end{align}
Since  $(123)=(13)(12)$ and $(132) = (12)(13)$, the last two identities can also be derived from $P_{n-j,j}^{(12);\k}$ and $P_{n-j,j}^{(13);\k}$, which 
explains our notation $\tau\k$, that is, writing the group action on the left side.   

\begin{lem}
For $0 \le j \le n$, 
\begin{align} \label{eq:Pk(12)}
P_{n-j,j}^{(12);\k} (1,x_2)= & (-1)^{n-j} \frac{(\k_2+1)_{n-j}(\k_3+1)_j}{(n-j)!j!} \sum_{m=0}^n
    \frac{(-n)_m (n+|\k|+2)_m}{m!(\k_2+1)_m} \\
     & \times {}_4F_3 \left (\begin{matrix} -m, m+\k_2+\k_3+1, -j, j+\k_1+\k_3+1 \\-n, \k_3+1, n+|\k|+2 \end{matrix}; 1 \right) x_2^m.
     \notag 
\end{align}
\end{lem}

\begin{proof}
By \eqref{Jacobi1} with $P_n^{(\a,\b)}(-t) = (-1)^n P_n^{(\b,\a)}(t)$ and \eqref{Jacobi2}, it is easy to see that
\begin{align*}
  P_{n-j,j}^{(12);\k} (1,x_2)= & (-1)^{n-j} \frac{(\k_2+1)_{n-j}(\k_3+1)_j}{(n-j)!j!} \\
  & \times  {}_2F_1 \left (\begin{matrix} -n+j, n+j+|\k|+2\\\k_2+1\end{matrix}; x_2 \right) 
      {}_2F_1 \left (\begin{matrix} -j, -j - \k_1 \\ \k_3+1\end{matrix}; x_2 \right). 
\end{align*}
Writing the product of the two ${}_2F_1$ as a single sum and rearranging the terms by using
$$
   \frac{(-j)_{m-i}}{(m-i)!} = \frac{(-j)_m(-m)_i}{m!(1-m+j)_i} \quad\hbox{and}\quad
       \frac{(-j-\k_1)_{m-i}}{(\k_3+1)_{m-i}} = \frac{(-j-\k_1)_m(-\k_3-m)_i}{(\k_3+1)_m(1-m+j+\k_1)_i},
$$
we see that the product of the two  ${}_2F_1$ is equal to 
\begin{align*}
   \sum_{m=0}^\infty  \frac{(-j)_m (-j -\k_1)_m}{m!(\k_3+1)_m}  
  {}_4F_3 \left (\begin{matrix} -n+j, n+j+|\k|+2, -m, -\k_3-m\\\k_2+1, 1-m+j, 1-m+\k_1+j \end{matrix}; 1 \right)x_2^m.
\end{align*}
We can think of each term in the above sum as a meromorphic function of $j$, which is analytic at the nonnegative integers. 
This ${}_4F_3$ can be rewritten by using the iterated Whipple identity \cite[(4.5)]{GI}
\begin{align*}
(u)_m(v)_m & (w)_m  {}_4F_3 \left (\begin{matrix} -m, x,y,z\\ u,v,w \end{matrix}; 1 \right) \\
  = \, & (1-x-m)_m(1-v+y-m)_m(1-v+z-m)_m \\
   & \times {}_4F_3 \left (\begin{matrix} -m, w-x,u-x,1-v-m\\ 1-x-m,1-v+y-m,1-v+z-m \end{matrix}; 1 \right)
\end{align*}
with $x = -j$, $y=m+\k_2+\k_3+1$, $z= j+\k_1+\k_3+1$, $u=-n$, $v=\k_3+1$, and $w=n+|\k|+2$, which
leads to the ${}_4F_3$ in \eqref{eq:Pk(12)} with a constant that can be combined with the constant in
front of the original ${}_4F_3$. Putting these together proves \eqref{eq:Pk(12)}.
\end{proof}

We are now ready to derive explicit formulas for the connection coefficients.

\begin{prop}
For  $0 \le j \le n$, the connection coefficients for $\tau = (12)$ satisfy
\begin{equation} \label{eq:C(12)}
 c_{j,m}^{(12)}(\k,n) = D_{j,m}^n(\k)  {}_4F_3 \left (\begin{matrix} -m, m+\k_2+\k_3+1, -j, j+\k_1+\k_3+1 \\-n, \k_3+1, n+|\k|+2 \end{matrix}; 1 \right),
\end{equation}
where the constant $D_{j,m}^n(\k)$ is given by
\begin{align} \label{eq:Ajm}
 D_{j,m}^n(\k) = & (-1)^{n+m} \frac{(-n)_j (\k_2+1)_{n-j}(\k_3+1)_j}{j! (\k_2+1)_m} \\
  & \times  \frac{ (n+|\k|+2)_m}{(\k_2+\k_3+2m+2)_{n-m}(\k_2+\k_3+m+1)_m}. \notag
\end{align}
\end{prop}

\begin{proof}
Setting $x_1 =1$ shows that $P_{n-m,m}^\k(1,x_2)$ is equal to a multiple of $x_2^m$, where the constant is determined
by $P_{n-j}^{(\k_2+\k_3+2j+1,\k_1)}(1)$ and the leading coefficient of $P_j^{(\k_3,\k_2)}$. More precisely, 
\begin{equation} \label{eq:Pk(1,x2)}
   P_{n-m,m}^\k(1,x_2) = \frac{(\k_2+\k_3+2m+2)_{n-m} (\k_2+\k_3+m+1)_m}{(n-m)! m!} x_2^m.
\end{equation}
Hence, the right hand side of \eqref{eq:con-coe-2d} when $x_1 =1$ is a polynomial of degree $n$ in $x_2$, so is the right hand 
side, by \eqref{eq:Pk(12)}, of \eqref{eq:con-coe-2d}. Comparing the coefficients of $x_2^m$ in the two sides proves the stated result. 
\end{proof}

The ${}_4F_3$ in the lemma is a Racah polynomial of degree $j$ in $m$ or a Racah polynomial of degree $m$ in $j$. More
precisely, let us state the connection coefficients in terms of the orthonormal Racah polynomials. 

\begin{cor} \label{cor:C(12)d=2}
For  $0 \le j \le n$, the connection coefficients for $\tau = (12)$ satisfy
$$
\wh  c_{j,m}^{(12)}(\k) = (-1)^{n+m+j}\sqrt{w_R(m; \sigma_1)} \wh R_j(\l(m); \s_1), \qquad  \l(m) = m (m+\k_2+\k_3+1)
$$
where $\s_1 = (-n-1, n+\k_1+\k_3+1, \k_3,\k_2)$, and 
$$
\wh  c_{j,m}^{(12)}(\k) = (-1)^{n+m+j}\sqrt{w_R(j; \sigma_2)} \wh R_m(\l(j); \s_2), \qquad  \l(j) = j (j+\k_1+\k_3+1)
$$
where $\s_2 = (-n-1, n+\k_2+\k_3+1, \k_3,\k_1)$.
\end{cor}

\begin{prop} \label{prop:c-c-2d}
For $0 \le j \le n$, let $c_{j,m}^\tau(\k) = c_{j,m}^\tau(\k, n)$. Then 
\begin{align}\label{eq:not(12)}
\begin{split}
 c_{j,m}^{(13)}(\k) &= (-1)^{m+j} c_{j,m}^{(12)} ((23)\k), \qquad  c_{j,m}^{(123)}(\k) = (-1)^{j} c_{j,m}^{(12)} (\k), \\
 c_{j,m}^{(23)}(\k) & =  (-1)^j \delta_{j,m}, \qquad c_{j,m}^{(132)} (\k) =(-1)^{j} c_{j,m}^{(13)} (\k). 
\end{split}
\end{align}
\end{prop}

\begin{proof}
By \eqref{eq:trivial-2d}, only the case $(13)$ needs a proof. Directly from the definition and using the fact that $P_j^{(\k_1,\k_2)}(\frac{-2 x_2}{1-x_2}-1) =(-1)^j P_j^{(\k_2,\k_1)}(\frac2{1-x_2}-1)$, it follows that 
$$
   P_{n-j,j}^{(13);\k}(1,-x_2) = (-1)^j  (23)_\k P_{n-j,j}^{(12);\k}(1,x_2),
$$
where $(23)_\k g(\k) = g((23)\k)$ and the subscript $\k$ indicates that the transposition is acting on $\k$. 
Moreover, \eqref{eq:Pk(1,x2)} shows that 
$$
  P_{n-m,m}^{\k}(1,-x_2) = (-1)^m P_{n-m,m}^{\k}(1,x_2)  
$$
is invariant under the action of $(23)$. Consequently, setting $x_1 =1$ and replacing
$x_2$ by $-x_2$ in \eqref{eq:con-coe-2d} and applying $(23)$ on $\k$, we obtain
$$
 (-1)^j P_{n-j,j}^{(12);\k}(1,x_2) = \sum_{m=0}^n c_{j,m}^{(13)}((23)\k) (-1)^m P_{n-m,m}^{\k}(1,x_2),
$$
which shows that $c_{j,m}^{(13)}((23)\k) = (-1)^{j+m} c_{j,m}^{(12)}(\k)$ by the definition of $c_{j,m}^{(12)}(\k)$.  
\end{proof}

Since $(12)(13) = (132)$, the relation \eqref{eq:c-convolution} can be used to derive a summation formula between
the Racah polynomials. Recall that the weight function for Racah polynomials is given in 
\eqref{eq:Racah-weight}. We define a related weight function $w^*$ by
$$
     w(x; \a,\b,\g, \g) = w^*(x;\a,\b,\g,\d) \frac{(\g+1)_x}{(\d+1)_x}.
$$
 
\begin{prop} \label{prop:4.5}
Let $u(x;\k,n): = w^*(x; -n-1,n+|\k|-\k_1+1,\k_3,\k_1)$. Then 
\begin{align} \label{eq:sum4F3}
  \sum_{m=0}^n (-1)^m u(m;\k,n)&  {}_4F_3 \left (\begin{matrix} -m, m+\k_1+\k_3+1, -k, k+\k_1+\k_2+1 \\
       -n, \k_1+1, n+|\k|+2 \end{matrix}; 1 \right)  \\
 \times &  {}_4F_3 \left (\begin{matrix} -m, m+\k_1+\k_3+1, -\ell, \ell+\k_2+\k_3+1 \\-n, \k_3+1, n+|\k|+2 \end{matrix}; 1 \right) \notag \\
  = (-1)^{n+k+\ell} & \frac{(\k_2+1)_k (\k_2+1)_\ell}{(\k_2+1)_n} \frac{(\k_1+\k_3+2)_n}{(\k_1+1)_k (\k_3+1)_\ell} \notag \\
  \times & {}_4F_3 \left (\begin{matrix} -k, k+\k_1+\k_2+1, -\ell, \ell+\k_2+\k_3+1 \\-n, \k_2+1, n+|\k|+2 \end{matrix}; 1 \right). \notag
\end{align}
\end{prop} 
 
\begin{proof}
To derive this identity, we use Proposition \ref{prop:convolu} with $\tau_1=(12)$ and $\tau_2=(13)$, so that 
\eqref{eq:c-convolution} becomes 
$$
  c_{k,\ell}^{(132)} (\k,n) =   \sum_{m=0}^n    c_{k,m}^{(13)} ((12)\k,n) c_{m,\ell}^{(12)} (\k,n).
$$
By \eqref{eq:C(12)} and \eqref{eq:not(12)}, all three connection coefficients in this identity are explicitly given in terms of 
${}_4F_3$ series, which are the three ${}_4F_3$ functions in the statement. We still need to check the coefficients in front of 
${}_4F_3$. For the sum in the right hand side, these coefficients are $(-1)^{k+m} D_{k,m}^n((12)(23)\k)D_{m,\ell}^n(\k)$ 
which can be split into a product of two terms, the first term is independent of $m$ and mostly cancels out with the terms in 
$D_{k,\ell}^n((23)\k)$ in the left hand side, whereas the second term is equal to 
$$
 \frac{(n+|\k|+2)_m}{(\k_1+\k_3+2m+2)_{n-m}(\k_1+\k_3+m+1)_m} \frac{ n!(\k_2+1)_{n-m} }{(n-m)!m!}
$$
which can be written as $u(m;\k,n) (\k_2+1)_n/(\k_1+\k_3+2)_n$ using the relations $(a)_{n-m} = (-1)^m (a)_n/(1-a-n)_m$
and $(a)_{2m} = 2^{2m} (\frac{a}2)_m(\f{a+1}2)_m$. We omit the details. 
\end{proof} 
 
One particular interesting case is $\k_1=\k_3$, for which the weight function becomes
$$
u(x;\k,n) = w(x; -n-1,n+|\k|-\k_1+1,\k_3,\k_1)
$$ 
and the two ${}_4F_3$ series in the summation of \eqref{eq:sum4F3} are Racah polynomials of degree $k$ and $\ell$, respectively, 
in the variable $m$ associated with this weigh function. Without the $(-1)^m$ in the summation, the left hand side of 
\eqref{eq:sum4F3} would be zero when $k \ne \ell$. 

\section{Jacobi polynomials of three variables}\label{se5}
\setcounter{equation}{0} 

In this section we consider the case $d=3$, which is complex enough to give us an idea of the complications in high dimensions and
also simple enough that we can have a clear picture for all permutations.  

In this case, the polynomials in \eqref{eq:Pnu} are of the form
\begin{align} \label{eq:OPd=3}
   P_{\nu}^\k(x_1,x_2,x_3) = &\ P_{\nu_1}^{(a_1,\k_1)}(2x_1-1) (1-x_1)^{\nu_2}
        P_{\nu_2}^{(a_2,\k_2)}\left (\frac{2 x_2}{1-x_1}-1\right) \\ 
         & \times (1-x_1-x_2)^{\nu_3}  P_{\nu_3}^{(\k_4,\k_3)}\left (\frac{2 x_3}{1-x_1-x_2}-1\right) \notag
\end{align}
for $|\nu| = \nu_1+\nu_2+\nu_3 = n$, where 
\begin{align*}
  a_1 & = a_1(\nu,\k) = \k_2+\k_3+\k_4+2\nu_2+2\nu_3 + 2,\\
  a_2 & = a_2(\nu,\k) = \k_3+\k_4+2\nu_3+1.
\end{align*}
Let $\tau$ be an element of $S_4$ and $x=(x_1,x_2,x_3) \in \RR^3$. By Definition~\ref{ppp}, $\tau x = (x_{\tau(1)}, x_{\tau(2)},x_{\tau(3)})$, 
where $x_4 = 1-x_1-x_2-x_3$. To simplify the notation, we define
$$
  P_{\nu}^{\tau; \k} (x_1,x_2,x_3) :=  P_{\nu}^{\tau \k}(x_{\tau(1)}, x_{\tau(2)},x_{\tau(3)}).  
$$
There are a total of 24 elements in $S_4$. Under permutations of $P_\mu^\k$, we 
end up with 12 different bases, which are given by 
\begin{equation}\label{eq:S4}
 (1), (12), (13), (14), (23), (24),  (123), (124), (132), (142), (13)(24), (14)(23),
\end{equation}
and bases derived from other permutations can be written in terms of them. Indeed, 
\begin{equation}\label{eq:swap}
\begin{alignedat}{2}
& P_\nu^{(34);\k}(x) = (-1)^{\nu_3} P_\nu^\k(x), && P_\nu^{(12)(34);\k}(x) = (-1)^{\nu_3} P_\nu^{(12);\k}(x),  \\
& P_\nu^{(134);\k}(x) = (-1)^{\nu_3} P_\nu^{(13);\k}(x), && P_\nu^{(143);\k}(x) = (-1)^{\nu_3} P_\nu^{(14);\k}(x), \\
& P_\nu^{(234);\k}(x) = (-1)^{\nu_3} P_\nu^{(23);\k}(x), && P_\nu^{(243);\k}(x) = (-1)^{\nu_3} P_\nu^{(24);\k}(x), \\
&  P_\nu^{(1234);\k}(x) = (-1)^{\nu_3} P_\nu^{(123);\k}(x), && P_\nu^{(1243);\k}(x) = (-1)^{\nu_3} P_\nu^{(124);\k}(x),  \\
& P_\nu^{(1324);\k}(x) = (-1)^{\nu_3}P_\nu^{(13)(24);\k}(x),\quad  && P_\nu^{(1342);\k}(x) = (-1)^{\nu_3} P_\nu^{(132);\k}(x), \\ 
& P_\nu^{(1423);\k}(x) = (-1)^{\nu_3}P_\nu^{(14)(23);\k}(x), \quad && P_\nu^{(1432);\k}(x) = (-1)^{\nu_3} P_\nu^{(142);\k}(x). 
\end{alignedat}
\end{equation}
The polynomials $P_\nu^\k$ possess inherited structures of two variables, which can be utilized to derive formulas for 
connection coefficients from those for orthogonal polynomials of two variables. We start with $c_{\nu,\mu}^{(12)}(\k)$.

\begin{prop}\label{prop:C(12)d=3}
For $|\nu| = |\mu| =n$, 
 \begin{align} \label{eq:C(12)d=3}
    c_{\nu,\mu}^{(12)}(\k)  = & \, \delta_{\nu_3,\mu_3} c_{\nu_2,\mu_2}^{(12)}(\wh \k, n- \nu_3) \qquad\hbox{with}\quad
           \wh \k =(\k_1,\k_2,\k_3+\k_4+2\nu_3+1) \\
           = & \, \delta_{\nu_3,\mu_3} D_{\nu_2,\mu_2}^{n-\nu_3} (\k_1,\k_2,\wh \k_3) \notag \\
        & \times {}_4F_3 \left (\begin{matrix} -\mu_2, \mu_2+\k_2+ \wh \k_3+1, -\nu_2, \nu_2 +\k_1+\wh \k_3+1 \\
              -n + \nu_3, \wh \k_3+1, n+\nu_3+|\k|+3 \end{matrix}; 1 \right)  \notag 
\end{align}
where $c_{\nu_2,\mu_2}^{(12)}(\wh \k,n-\nu_3)$ are the connection coefficients in \eqref{eq:C(12)}. Furthermore,
\begin{align} \label{eq:C(12)(34)} 
     c_{\nu,\mu}^{(34)}(\k) = (-1)^{\nu_3} \delta_{\nu,\mu}, \qquad c_{\nu,\mu}^{(12)(34)}(\k) =  (-1)^{\nu_3} c_{\nu,\mu}^{(12)}(\k).
\end{align}
\end{prop}

\begin{proof}
With $\wh \k_3 = a_2(\k,\nu) = \k_3+\k_4 + 2\nu_3+1$ and $\wh \k = (\k_1,\k_2,\wh \k_3)$, we can write 
$$
  P_\nu^\k(x) = P_{\nu_1,\nu_2}^{\wh \k}(x_1,x_2)(1-x_1-x_2)^{\nu_3}  P_{\nu_3}^{(\k_4,\k_3)}\left (\frac{2 x_3}{1-x_1-x_2}-1\right),
$$
where $P_{\nu_1, \nu_2}^{\wh \k} = P_{n-\nu_3- \nu_2, \nu_2}^{\wh \k}$ is the orthogonal polynomial in two variables given 
by \eqref{eq:Pjn}. Furthermore, the transposition $(12)$ gives 
\begin{align*}
   P_{\nu}^{(12);\k}(x) = P_{\nu_1,\nu_2}^{(12);\wh \k}(x_1,x_2)(1-x_1-x_2)^{\nu_3}
         P_{\nu_3}^{(\k_4,\k_3)}\left (\frac{2 x_3}{1-x_1-x_2}-1\right),
\end{align*}
which has a common factor, the one indexed by $\nu_3$, with $P_\nu^\k$ and, as a consequence, it is easy to see that 
$c_{\nu,\mu}^{(12)}(\k) = 0$ if $\nu_3 \ne \mu_3$. If $\nu_3 = \mu_3$, then the connection coefficients clearly reduce to 
those of $c_{\nu_2,\mu_2}^{(12)}(\wh \k, n- \nu_3)$ for orthogonal polynomials in two variables, which proves 
\eqref{eq:C(12)d=3}. 

The other two cases follow readily from the first two identities in \eqref{eq:swap}. The second one also follows from
Proposition \ref{prop:convolu}.
\end{proof}
    
The polynomial $P_\nu^\k$ has another structure of two variables, for which we start with a simple observation. 

\begin{lem} \label{lem:2d-in-3d}
For $\tau$ in the subgroup $\{(1), (23), (24), (34), (234), (243)\}$, the connection coefficients satisfy
\begin{equation}\label{eq:2d-in-3d}
   c_{\nu,\mu}^{\tau} (\k) = \delta_{\nu_1, \mu_1} c_{\nu_3,\mu_3}^{\tau} (\wh \k, n- \nu_1), \qquad  \wh \k:= (\k_2,\k_3,\k_4) 
\end{equation}
where $c_{\nu_3,\mu_3}^\tau (\wh \k, n-\nu_1)$ are the connection coefficients for orthogonal polynomials on the triangle for
the variables $(y_2,y_3, 1-y_2-y_3)$.
\end{lem}

\begin{proof}
Setting $y_2 = \frac{x_2}{1-x_1}$ and $y_3 = \frac{x_3}{1-x_1}$, so that $1-y_2-y_3 = \frac{1-x_1-x_2-x_3}{1-x_1}$,  
it is easy to see that we can rewrite $P_\nu^\k$ as
\begin{equation} \label{Pnu-in-y}
P_\nu^\k (x) = P_{\nu_1}^{(a_1,\k_1)}(2x_1-1)(1-x_1)^{n- \nu_1} P_{\nu_2, \nu_3}^{\wh \k} (y_2,y_3), \quad  
\end{equation}
where $P_{\nu_2, \nu_3}^{\wh \k} = P_{n-\nu_1- \nu_3, \nu_3}^{\wh \k}$ is the orthogonal polynomial in two variables given 
by \eqref{eq:Pjn}. Together with the fact that $a_1(\nu,\k) = |\wh \k|+ 2\nu_2 +2\nu_3+2$ is invariant under permutations in the 
subgroup, it follows that $\tau$ in the subgroup only acts on $P_{\nu_2, \nu_3}^{\wh \k}$. As a result, we conclude that  
$c_{\nu,\mu}^\tau (\k) = 0$ if $\nu_1 \ne \mu_1$. If $\nu_1 = \mu_1$, then \eqref{Pnu-in-y} shows that the connection 
coefficients can be derived from those that express $P_{\nu_2,\nu_3}^{\tau;\k}(y_2,y_3)$ in terms of $P_{\mu_2,\mu_3}^\k(y_2,y_3)$ in two variables. 
\end{proof}

\begin{rem}
In the right hand side of \eqref{eq:2d-in-3d}, $\tau$ is acting on $(\k_2,\k_3,\k_3)$ as a subset of $(\k_1,\k_2,\k_3,\k_4)$,
not on the positions of the elements in $(\k_2,\k_3,\k_4)$. For example, $(23)(\k_2,\k_3,\k_4) = (\k_3,\k_2,\k_4)$.
\end{rem}

As a consequence, we can derive from \eqref{eq:C(12)} and \eqref{eq:not(12)} the following corollary: 
\begin{cor}
For $|\nu| = |\mu| =n$, 
\begin{align} \label{eq:C(23)d=3}
  c_{\nu,\mu}^{(23)}(\k) = &\, \delta_{\nu_1,\mu_1} D_{\nu_3,\mu_3}^{n-\nu_1}(\k_2,\k_3,\k_4) \\
        & \times {}_4F_3 \left (\begin{matrix} -\mu_3, \mu_3+\k_3+\k_4+1, -\nu_3, \nu_3 +\k_2+\k_4+1 \\
              -n + \nu_1, \k_4+1, n-\nu_1+|\k|-\k_1+2 \end{matrix}; 1 \right)  \notag 
\end{align}
and 
\begin{align} \label{eq:C(24)d=3}
 c_{\nu,\mu}^{(24)}(\k) = (-1)^{\nu_3+\mu_3} c_{\nu,\mu}^{(23)}((34)\k). 
\end{align}
\end{cor}
Furthermore, the identities in \eqref{eq:swap} show that 
$$  
c_{\nu,\mu}^{(234)}(\k) = (-1)^{\nu_3} c_{\nu,\mu}^{(23)}(\k) \quad \hbox{and}\quad  
     c_{\nu,\mu}^{(243)}(\k) = (-1)^{\nu_3} c_{\nu,\mu}^{(24)}(\k). 
$$

In the cases that we deal with so far, the connection coefficients are Racah polynomials just like in the case of 
two variables. This will change in our next case, for which it is more convenient to give the formula for the normalized
connection coefficients. 

\begin{prop} \label{prop:C(123)d=3}
For $|\nu| = n$ and $|\mu|=n$, 
\begin{align} \label{eq:C(123)d=3}
 \wh c_{\nu,\mu}^{(123)}(\k) =  (-1)^{n+\nu_3} \sqrt{w_\sR(x; \beta, n)}\, \wh \sR_{(\mu_3,\mu_2)}( x; \beta, n),
\end{align}
where $\beta = (\k_1, \k_1+\k_4+1,\k_1+\k_3+\k_4+2, |\k|+3)$ and $x=(\nu_3, \nu_2+\nu_3)$. Moreover, 
\begin{align} \label{eq:C(132)d=3}
 \wh c_{\nu,\mu}^{(132)}(\k) = \wh c_{\mu,\nu}^{(123)}((132)\k)= (-1)^{n+\mu_3}
   \sqrt{w_\sR(y; \beta^*, n)} \, \wh \sR_{(\nu_3,\nu_2)}(y; \beta^*, n),
\end{align}
where $\beta^* = (\k_3, \k_3+\k_4+1,\k_2+\k_3+\k_4+2, |\k|+3)$ and $y=(\mu_3,\mu_2+\mu_3)$.
\end{prop}

\begin{proof}
Since $(123) = (12)(23)$, we use Proposition \ref{prop:convolu}, which shows that 
$$
 c_{\nu,\mu}^{(123)}(\k) = \sum_{|\omega| = n} c_{\nu,\omega}^{(23)} ((12)\k) c_{\omega,\mu}^{(12)}(\k).  
$$
Since, by \eqref{eq:C(23)d=3}, $c_{\nu,\omega}^{(23)} ((12)\k)$ contains $\delta_{\nu_1,\omega_1}$ and, 
by \eqref{eq:C(12)d=3}, $c_{\omega,\mu}^{(12)} (\k)$ contains $\delta_{\omega_3,\mu_3}$, it follows that 
$$
  c_{\nu,\mu}^{(123)}(\k) =  c_{\nu,\omega^*}^{(23)} ((12)\k) c_{\omega^*,\mu}^{(12)}(\k) \quad\hbox{with} \quad
    \omega^* = (\nu_1,n-\nu_1-\mu_3,\mu_3).
$$ 
Hence, it follows from  \eqref{eq:C(23)d=3} and \eqref{eq:C(12)d=3} that 
\begin{align*}
  &  c_{\nu,\mu}^{(123)}(\k) =  D_{\nu_3,\mu_3}^{n-\nu_1}(\k_1,\k_3,\k_4) D_{n-\nu_1-\mu_3,\mu_2}^{n-\mu_3}
       (\k_1,\k_2,\k_3+\k_4+2\mu_3+1) \\
  &  \times {}_4F_3 \left (\begin{matrix} -\mu_3, \mu_3+\k_3+\k_4+1, -\nu_3, \nu_3 +\k_1+\k_4+1 \\
              -n + \nu_1, \k_4+1, n-\nu_1+|\k|-\k_2 + 2 \end{matrix}; 1 \right)  \notag \\
  &  \times {}_4F_3 \left (\begin{matrix} -\mu_2, \mu_2+|\k|-\k_1+2\mu_3+2, -n+\nu_1+\mu_3, n-\nu_1+\mu_3+|\k|-\k_2+2 \\
              -n + \mu_3, \k_3+\k_4+2\mu_3+2, n+\mu_3+|\k| +3 \end{matrix}; 1 \right).  \notag        
\end{align*}

Comparing with the explicit formula, it is easy to see that the two ${}_4F_3$ functions are exactly those two that appear 
in $\sR_\mu(x;\b,n)$ with $x$ and $\beta$ as specified in the statement. The explicit formula of $\sR_\mu(x;\b,n)$
also contains a $(- x_2)_{\mu_3}(|\k|-\k_2+ x_2+2)_{\mu_3}$. As a result it is sufficient to show that 
$$
\Delta:= \frac{D_{\nu_3,\mu_3}^{n-\nu_1}(\k_1,\k_3,\k_4) D_{n-\nu_1-\mu_3,\mu_2}^{n-\mu_3}(\k_1,\k_2,\k_3+\k_4+2\mu_3+1)}
  {(-n+\nu_1)_{\mu_3}(|\k|-\k_2+n-\nu_1+2)_{\mu_3} \sqrt{A_\nu(\k(123)) w_{\sR}(x,\beta,n)}}
     = r_\mu(\b,n),
$$
where $x = (\nu_3,n-\nu_1)$. By \eqref{eq:discreteOP} and the orthogonality of the Racah polynomials, it is 
sufficient to show that $\Delta$ is independent of $x$. This involves tedious computation. For the record let us
write down the intermediate steps.  
\begin{align*}
 A_\nu& ((123)\k) \equiv  \frac{(\k_2+1)_{\nu_1}(\k_3+1)_{\nu_2}(\k_4+1)_{\nu_3}(\k_1+1)_{\nu_3}}
    {\nu_1!\nu_2!\nu_3! (|\k|-\k_2+2n-2\nu_1+2)(\k_1+\k_4+2\nu_3+1)} \\
     &\times \frac{(|\k|-\k_2+3)_{2n-\nu_1}(\k_1+\k_4+2)_{n-\nu_1+\nu_3}}
            {(|\k|+3)_{2n-\nu_1}(|\k|-\k_2+2)_{n-\nu_1+\nu_3}(\k_1+\k_4+1)_{\nu_3}},
\end{align*}
where $\nu_2 = n-\nu_1-\nu_3$ and $\equiv$ means that the equality holds up to a multiple constant that is independent
of $\nu$. Using this formula and the explicit formula of $w_\sR(x;\beta,n)$, we can verify that 
$$
  \sqrt{A_\nu((123)\k) w_{\sR}(x,\beta,n)} \equiv \frac{(\k_2+1)_{\nu_1}(\k_4+1)_{\nu_3}
      (\k_3+1)_{n-\nu_1-\nu_3}}{\nu_1!\nu_3! (n-\nu_1-\nu_3)!}.
$$
The formula of $D_{\nu_3,\mu_3}^{n-\nu_1}(\k_1,\k_3,\k_4) D_{n-\nu_1-\mu_3,\mu_2}^{n-\mu_3}
(\k_1,\k_2,\k_3+\k_4+2\mu_3+1)$ can be derived from \eqref{eq:Ajm}, which can be used to show, together with
the above formula, that $\Delta$ is independent of $\nu$. We omit the details. 

Finally, since $(132) = (123)^{-1}$, the formula for $\wh c_{\nu,\mu}^{(132)}(\k)$ follows from the one for  
$\wh c_{\nu,\mu}^{(123)}(\k)$ by equation \eqref{eq:inverse}. 
\end{proof}

Recall that  $[\wh c_{\nu,\mu}^{(123)}(\k)]_{\nu,\mu}$ is an orthogonal matrix. Formula \eqref{eq:C(123)d=3} naturally connects the second orthogonality in \eqref{eq:discreteOP} to the orthogonality of the Racah polynomials. The first orthogonality relation in \eqref{eq:discreteOP} is more subtle and can be related again to the Racah polynomials via the duality established in Proposition~\ref{prop:duality}. We formulate this precisely in arbitrary dimension in Theorem~\ref{thm:C(cyclic)}.

Explicit formulas for connection coefficients corresponding to other permutations can be derived from those of 
$\wh c_{\nu,\mu}^{(123)}(\k)$ and $\wh c_{\nu,\mu}^{(132)}(\k)$. 

\begin{prop}
For $|\nu| = n$ and $|\mu|=n$, 
\begin{equation} \label{eq:C(124)d=3}
\begin{alignedat}{2}
 & \wh c_{\nu,\mu}^{(124)}(\k) = (-1)^{\nu_3+\mu_3}  \wh c_{\nu,\mu}^{(123)}((34)\k), \quad
  && \wh c_{\nu,\mu}^{(142)}(\k) = (-1)^{\nu_3+\mu_3}  \wh c_{\nu,\mu}^{(132)}((34)\k), \\
 & \wh c_{\nu,\mu}^{(1234)}(\k) = (-1)^{\nu_3}  \wh c_{\nu,\mu}^{(123)}(\k),  
   && \wh c_{\nu,\mu}^{(1342)}(\k) = (-1)^{\nu_3} \wh c_{\mu,\nu}^{(132)}(\k), \\ 
 & \wh c_{\nu,\mu}^{(1243)}(\k) = (-1)^{\mu_3}  \wh c_{\nu,\mu}^{(123)}((34)\k), 
   &&  \wh c_{\nu,\mu}^{(1432)}(\k) = (-1)^{\mu_3}  \wh c_{\nu,\mu}^{(132)}((34)\k).
\end{alignedat}
\end{equation}
\end{prop}

\begin{proof}
Since $(34)(123)(34) = (124)$ and  $(34)(132)(34) = (142)$, the two identities in the first line follow from 
\eqref{eq:c-convolution} and $c_{\nu,\mu}^{(34)} (\k) = (-1)^{\mu_3} \delta_{\nu,\mu}$. The other identities 
follow from the relations in \eqref{eq:swap}. 
\end{proof}

All these connection coefficients are given in terms of the Racah polynomials of two variables. 
The remaining permutations are $(13), (14)$, $(13)(24)$, $(14)(23)$, $(134)$, $(143)$, $(1324)$, $(1423)$. 
By \eqref{eq:swap}, we only need to deal with the first four. Let us consider, for example, $(13)$. 

\begin{prop}
For $|\nu| = n$ and $|\mu|=n$, 
\begin{align} \label{eq:C(13)d=3}
  \wh c_{\nu,\mu}^{(13)}(\k)   =   (-1)^{\nu_2} 
     \sum_{\ell = 0}^{n-\nu_3}   b_{\nu,\mu}(\ell) \wh R_{\nu_2}(\l(\ell);\sigma)
        \wh \sR_{(\mu_3,\mu_2)}( (\nu_3, \ell + \nu_3); \beta,n),  
\end{align}
where  $b_{\nu,\mu}(\ell) = (-1)^{\ell}\sqrt{w_R(\ell;\sigma)} \sqrt{w_\sR((\nu_3, \ell+\nu_3),\beta,n)}$ with
\begin{align*}
  \sigma & = (-n+\nu_3-1,n+\nu_3+|\k|-\k_3+2, \k_1+\k_4+2\nu_3+1,\k_3), \\
   \beta & = (\k_1,\k_1+\k_4+1,\k_1+\k_3+\k_4+2, |\k|+3). 
\end{align*}
\end{prop}

\begin{proof}
Since $(13) = (123)(12)$, it follows from \eqref{eq:c-convolution} that 
\begin{align*}
  \wh c_{\nu,\mu}^{(13)}(\k)  & = \sum_{|\omega|=n}\wh c_{\nu,\omega}^{(12)}((123)\k) \wh c_{\omega,\mu}^{(123)}(\k) \\
   & = (-1)^{\nu_2}\sum_{\omega_1+\omega_2 = n-\nu_3}   b_{\nu,\mu}(\omega_2) \wh R_{\nu_2}(\l(\omega_2);\sigma)
        \wh \sR_{(\mu_3,\mu_2)}( (\nu_3,n-\omega_1); \beta,n),
\end{align*}
which can be written as a sum over $\ell = \omega_2$, since $n-\omega_1 = \omega_2+\nu_3$. 
\end{proof}

\begin{rem} \label{rem:5.7}
Using \eqref{eq:c-convolution} and explicit formulas for connection coefficients, we can also derive various relations for 
Racah polynomials of one and two variables in the spirit of Proposition \ref{prop:4.5}. For example, using $(12)(123) = (23)$,
we end up with a formula that expresses a Racah polynomial of one variable as a sum of a product of Racah polynomials of 
two variables and Racah polynomials of one variable. 
\end{rem}

\begin{rem}
Note that \eqref{eq:C(13)d=3} resembles the formula for the Wigner 
$9j$ symbols as a sum over triple products of $6j$ symbols. It would be 
interesting to explore this connection in more depth and to relate the results 
in the present paper to Lie theory along the lines of the works \cite{LVdJ,Ros}.
\end{rem}

\section{Jacobi polynomials of $d$ variables}\label{se6}
\setcounter{equation}{0} 

Most of the explicit formulas in the previous section can be extended to arbitrary dimension. For $j\in\NN$, let $\wh S^{j}_{d+1}$ denote the subgroup of $S_{d+1}$ fixing the points of $\{1,2,\dots,j\}$. Then similarly to Lemma~\ref{lem:2d-in-3d}, we see that the connection coefficients for $\tau\in \wh S^{j}_{d+1}$ can be computed from the connection coefficients of Jacobi polynomials in $d-j$ variables.

\begin{prop}\label{prop:fix-first}
For $\tau\in \wh S^{j}_{d+1}$ we have
\begin{equation}\label{eq:fix-first}
c_{\nu,\mu}^{\tau}(\k)=\delta_{\bnu_{j},\bmu_{j}}\,c_{\bnu^{j+1},\bmu^{j+1}}^{\tau}(\bka^{j+1}, n-|\bnu_j|), 
\end{equation}
where $c_{\bnu^{j+1},\bmu^{j+1}}^{\tau}(\bka^{j+1}, n-|\bnu_j|)$ are the connection coefficients for Jacobi polynomials on the simplex $T^{d-j}$ in the variables $y_{j+1},y_{j+2},\dots,y_{d}$ with parameters $\bka^{j+1}=(\kappa_{j+1},\kappa_{j+2},\dots,\kappa_{d+1})$. Moreover, formula \eqref{eq:fix-first} holds also for the normalized connection coefficients.
\end{prop}

\begin{proof}
The proof of  \eqref{eq:fix-first} is similar to the proof of Lemma~\ref{lem:2d-in-3d}. The fact that  \eqref{eq:fix-first} holds also for the normalized connection coefficients follows by using the explicit formula for the norms of the Jacobi polynomials  \eqref{eq:norm1} and the connection between $c_{\nu,\mu}^{\tau}(\k)$ and $\wh c_{\nu,\mu}^{\tau}(\k)$ given in equation \eqref{eq:conn-c-coeff}.
\end{proof}

For $k<d+1$ we think of $S_{k}$ as the subgroup of $S_{d+1}$ which acts on the first $k$ elements while keeping $\{k+1,k+2,\dots,d+1\}$ 
fixed. Similarly to Proposition~\ref{prop:C(12)d=3} we can reduce the computation of $c_{\nu,\mu}^{\tau}(\k)$ for  
$\tau\in S_{k}\subset S_{d+1}$ to the computation of the connection coefficients in lower dimension as follows.

\begin{prop}\label{prop:fix-last}
For $k<d$ and $\tau\in S_{k}$ we have
\begin{equation}\label{eq:fix-last}
c_{\nu,\mu}^{\tau}(\k)=\delta_{\bnu^{k+1},\bmu^{k+1}}\,c_{\bnu_{k},\bmu_{k}}^{\tau}(\hat \kappa, n-|\bnu^{k+1}|),  
\end{equation}
where $c_{\bnu_{k},\bmu_{k}}^{\tau}(\hat \kappa, n-|\bnu^{k+1}|)$  are the connection coefficients for Jacobi polynomials on the simplex $T^{k}$ in the variables $x_{1},x_{2},\dots,x_{k}$ with parameters 
$$\hat \kappa=(\kappa_1,\kappa_2,\dots,\kappa_{k},|\bka^{k+1}|+2|\bnu^{k+1}|+d-k).$$  Moreover, formula \eqref{eq:fix-last} holds also for the normalized connection coefficients.
\end{prop}

For the cyclic permutation we obtain the following explicit formula in terms of the Racah polynomials extending Corollary~\ref{cor:C(12)d=2} and Proposition~\ref{prop:C(123)d=3}.

\begin{thm} \label{thm:C(cyclic)}
For $\nu,\mu \in \NN_0^d$ with $|\nu| = n$ and $|\mu|=n$, 
\begin{align} \label{eq:C(cyclic)}
 \wh c_{\nu,\mu}^{(12\dots d)}(\k) =  (-1)^{n+\nu_{d}}  \sqrt{w_\sR(\hat \nu; \beta, n)}\, \wh \sR_{\hat \mu}(\hat \nu; \beta, n)
\end{align}
where $\wh \sR_{\hat \mu}(\hat \nu; \beta, n)$ are the orthonormal Racah polynomials of $d-1$ variables with indices $\hat\mu=(\mu_d,\mu_{d-1},\dots,\mu_2)$, variables $\hat\nu=(|\bnu^d|,|\bnu^{d-1}|,\dots,|\bnu^{2}|)$  and parameters $\beta_{j}=\kappa_1+|\bka^{d+2-j}|+j$ for $j=0,1,\dots,d$. We also have
\begin{align} \label{eq:C(cyclic)2}
 \wh c_{\nu,\mu}^{(12\dots d)}(\k) =  (-1)^{n+\nu_{d}}  \sqrt{w_\sR(\mut; \bet, n)}\, \wh \sR_{\nut}(\mut; \bet, n),
\end{align}
where $\wh \sR_{\nut}(\mut; \bet, n)$ are the orthonormal Racah polynomials of $d-1$ variables with indices $\nut=(\nu_1,\dots,\nu_{d-1})$, variables $\mut=(|\bmu_1|,|\bmu_2|,\dots,|\bmu_{d-1}|)$ and parameters $\bet_{0}=\k_1$, $\bet_{j}=-|\bka^{j+1}|-2n-d+j$ for $j=1,\dots,d$.
\end{thm}

\begin{proof} The proof of \eqref{eq:C(cyclic)} can be done by induction similarly to the proof of Proposition~\ref{prop:C(123)d=3}. We outline the main steps below. Suppose that the statement is true in any dimension less or equal to $d$ and we want to establish it in dimension $d+1$. We can decompose the cyclic permutation as $(12\dots d+1)=\tau_1\tau_2\in S_{d+2}$, where $\tau_1=(12\dots d)$ and $\tau_2=(d,d+1)$. By Proposition~\ref{prop:convolu} we have
\begin{align} \label{eq:cycl1}
\wh c_{\nu,\mu}^{(12\dots d+1)}(\k) = \sum_{|\omega|=n}\wh c_{\nu,\omega}^{(d,d+1)} ((12\dots d)\k) \wh c_{\omega,\mu}^{(12\dots d)}(\k).
\end{align}
Applying Proposition~\ref{prop:fix-first} with $\tau_2=(d,d+1)\in \wh S^{d-1}_{d+2}$ we see that 
\begin{equation*}
\wh c_{\nu,\omega}^{(d,d+1)}(\k)=\delta_{\bnu_{d-1},\bom_{d-1}}\,\wh c_{\bnu^{d},\bom^{d}}^{(d,d+1)}(\bka^{d}, n-|\bnu_{d-1}|).
\end{equation*}
Using the last formula and the induction hypothesis (or equivalently, Corollary~\ref{cor:C(12)d=2} and equations \eqref{eq:connection}) we see that
\begin{align} \label{eq:cycl2}
\wh c_{\nu,\omega}^{(d,d+1)}((12\dots d)\k)& =   \delta_{\bnu_{d-1},\bom_{d-1}}\,(-1)^{\nu_{d}} \\
 & \times \sqrt{w_\sR(\nu_{d+1}; \beta', \nu_d+\nu_{d+1})}\, \wh \sR_{\omega_{d+1}}(\nu_{d+1}; \beta',  \nu_d+\nu_{d+1}), \notag
\end{align}
where $\beta'=(\k_1,\k_1+\k_{d+2}+1,\k_1+\k_{d+1}+\k_{d+2}+2)$.
Applying now Proposition~\ref{prop:fix-last} with $\tau_1=(12\dots d)\in S_{d}\subset S_{d+2}$ we deduce that 
\begin{equation} \label{eq:cycl3}
\wh c_{\omega,\mu}^{(12\dots d)}(\k)=\delta_{\omega_{d+1},\mu_{d+1}}\,\wh c_{\bom_{d},\bmu_{d}}^{(12\dots d)}(\hat \kappa, n-\mu_{d+1}),  
\end{equation}
where $\hat \kappa=(\kappa_1,\kappa_2,\dots,\kappa_{d},\kappa_{d+1}+\kappa_{d+2}+2\mu_{d+1}+1)$.
From the last two equations it is clear that the sum in \eqref{eq:cycl1} can have only one nonzero term corresponding to 
$\omega$ with coordinates
\begin{equation*}
\omega_j=\begin{cases} 
\nu_j & \text{ if }j\leq d-1,\\
\nu_{d}+\nu_{d+1}-\mu_{d+1} & \text{ if }j=d,\\
\mu_{d+1} & \text{ if }j=d+1.
\end{cases}
\end{equation*}
In the rest of the proof we fix $\omega$ as above and therefore formula \eqref{eq:cycl1} becomes
\begin{align} \label{eq:cycl4}
\wh c_{\nu,\mu}^{(12\dots d+1)}(\k) = \wh c_{\nu,\omega}^{(d,d+1)} ((12\dots d)\k) \wh c_{\omega,\mu}^{(12\dots d)}(\k).
\end{align}
By the induction hypothesis we can rewrite formula \eqref{eq:cycl3} as
\begin{equation} \label{eq:cycl5}
\wh c_{\omega,\mu}^{(12\dots d)}(\k)=(-1)^{|\bmu_{d}|+\omega_d}
\sqrt{w_\sR(\nu''; \beta'', n-\mu_{d+1})}\, \wh \sR_{\mu''}(\nu''; \beta'', n-\mu_{d+1}),  
\end{equation}
where $\mu''=(\mu_d,\mu_{d-1},\dots,\mu_2)$, $\nu''=(|\bnu^{d}|-\mu_{d+1},|\bnu^{d-1}|-\mu_{d+1},\dots,|\bnu^2|-\mu_{d+1})$, 
$\beta''_0=\k_1$, $\beta''_j=\k_1+|\bka^{d+2-j}|+2\mu_{d+1}+1+j$, $j=1,2,\dots,d$.
Plugging \eqref{eq:cycl2}  and \eqref{eq:cycl5} in \eqref{eq:cycl4}, we see that $\wh c_{\nu,\mu}^{(12\dots d+1)}(\k)$ contains 
a product of $d$ ${}_4F_3$-series and it is not hard to show, using the explicit formulas above, that they coincide with the product 
of the ${}_4F_3$-series defining the $d$ dimensional Racah polynomial on the right-hand side in \eqref{eq:C(cyclic)} for the 
cyclic permutation $(12\dots d+1)$. Moreover, if $w_\sR(\hat \nu, \beta, n) \vert_{d\mapsto d+1}$ denotes the $w_\sR(\hat \nu, \beta, n)$ 
with $d$ replaced by $d+1$, 
one can check that
\begin{align*}
\frac{w_\sR(\hat \nu, \beta, n) \vert_{d\mapsto d+1}}{w_\sR(\nu_{d+1}; \beta', \nu_d+\nu_{d+1})\,w_\sR(\nu''; \beta'', n-\mu_{d+1})}
\end{align*}
can be rewritten as 
\begin{align*}
\frac{(\nu_{d}+\nu_{d+1}-\mu_{d+1})!\,(\k_1+1)_{\nu_{d}+\nu_{d+1}-\mu_{d+1}}}{(\k_{d+1}+\k_{d+2}+2)_{\nu_{d}+\nu_{d+1}+\mu_{d+1}}\,(\k_1+\k_{d+1}+\k_{d+2}+2)_{\nu_{d}+\nu_{d+1}+\mu_{d+1}}},
\end{align*}
up to a positive factor independent of $\nu$. Finally, one can show that the square of the norm 
$||\sR_{\mu_{d+1}}(\nu_{d+1}; \beta',  \nu_d+\nu_{d+1})||^2$ coincides with the inverse of the right-hand side in the last formula, up to a factor independent of $\nu$, and therefore 
\begin{align*}
\frac{w_\sR(\hat \nu, \beta, n) \vert_{d\mapsto d+1}\, ||\sR_{\mu_{d+1}}(\nu_{d+1}; \beta',  \nu_d+\nu_{d+1})||^2}{w_\sR(\nu_{d+1}; \beta', \nu_d+\nu_{d+1})\,w_\sR(\nu''; \beta'', n-\mu_{d+1}) }
\end{align*}
is independent of $\nu$. This combined with the second orthogonality relation in \eqref{eq:discreteOP} proves formula  \eqref{eq:C(cyclic)} for $(12\dots d+1)$, completing the induction. Finally, equation~\eqref{eq:C(cyclic)2} follows from \eqref{eq:C(cyclic)} and  Corollary~\ref{cor:dualorthogonality}.
\end{proof}

\begin{rem} Combining Theorem~\ref{thm:C(cyclic)} with Remark~\ref{rem:Racah2}, we also have  
\begin{align} \label{eq:C(cyclic)3}
 \wh c_{\nu,\mu}^{(12\dots d)}(\k) =  (-1)^{n+\nu_{d}}  \sqrt{w_\sR(\mu'; \beta', n)}\, \wh \sR'_{\nu'}(\mu'; \beta', n)
\end{align}
where $\wh \sR'_{\nu'}(\mu'; \beta', n)$ is the second family of orthonormal Racah polynomials of $d-1$ variables with indices $\nu'=(\nu_{d-1},\nu_{d-2},\dots,\nu_1)$, variables $\mu'=(|\bmu^d|,|\bmu^{d-1}|,\dots,|\bmu^{2}|)$ and parameters $\beta_{j}'=|\bka^{d+1-j}|+j$ for $j=0,1,\dots,d-1$, $\beta_{d}'=-2n-\k_1$.
\end{rem}

\begin{rem} Note that $P^{(d,d+1);\k}_{\nu}(x)=(-1)^{\nu_{d}}P^{\k}_{\nu}(x)$. 
Using Propositions \ref{prop:fix-first}, \ref{prop:fix-last} and Theorem~\ref{thm:C(cyclic)} we see that for every cyclic permutation of the form $\tau=(j,j+1,\dots,k-1,k)$, the corresponding connection coefficients can be written explicitly in terms of appropriate Racah polynomials in $k-j$ variables when $k\leq d$ and $d-j$ variables when $k= d+1$. For instance, for the adjacent transpositions $(j,j+1)$ we obtain the following explicit formulas:
\begin{align*}
\wh c_{\nu,\mu}^{(j,j+1)}(\k)=(-1)^{\mu_j+\nu_{j+1}}\delta_{\bnu_{j-1},\bmu_{j-1}}\,\delta_{\bnu^{j+2},\bmu^{j+2}}\,
\sqrt{w_R(\nu_{j+1}; \s)} \wh R_{\mu_{j+1}}(\l(\nu_{j+1}); \s) 
\end{align*}
where $\s=(-\nu_{j}-\nu_{j+1}-1,|\bka^{j+1}|+|\bnu^{j}|+|\bnu^{j+2}|+d-j, |\bka^{j+2}|+2|\bnu^{j+2}|+d-j-1,\k_j)$ for $j=1,\dots,d-1$, and 
\begin{align*}
\wh c_{\nu,\mu}^{(d,d+1)}(\k)=(-1)^{\nu_{d}}\delta_{\nu,\mu}.
\end{align*}
For more complicated permutations, we can apply Proposition~\ref{prop:convolu} and express the coefficients as sums of products of Racah polynomials. Furthermore, Remark \ref{rem:5.7} applies in $d$ dimensional setting. 
\end{rem}

\section{Hahn polynomials of several variables}\label{se7}
\setcounter{equation}{0}

Let $N$ be a positive integer. We use homogeneous coordinates of $\ZZ^{d+1}$, where
$$
\ZZ_N^{d+1} := \{\a: \a \in \NN_0^{d+1}, |\a| = N\}.
$$
For $\k \in \RR^{d+1}$ with $\k_i > -1$, $1 \le i \le d+1$, the Hahn weight function is defined by 
 \begin{equation}\label{eq:HK-homo}
  \sH_{\k,N} (x) =   \frac{(\k + \one)_x}{x!},  \qquad   x \in \ZZ_N^{d+1},
\end{equation}
where $\one=(1,\dots,1)\in\NN^{d+1}$. The Hahn polynomials are orthogonal with respect to the inner product  
\begin{equation}\label{eq:sum-homo}
     \la f, g\ra_{\sH_{\k,N}} =  \frac{N!}{(|\k|+ d+1)_N}\sum_{\a \in \ZZ_N^{d+1} } f(\a) g(\a) \sH_{\k,N}(\a).
\end{equation}
The Jacobi polynomials on the simplex can be used as a generating function for one orthogonal basis of the Hahn polynomials.
Let $P_\nu^\k$ be defined as in \eqref{eq:Pnu} and let 
$$
 p_\nu^\k: = \prod_{j=1}^d P_j^{(a_j(\k,\nu), \k_j)}(1) = \prod_{j=1}^d \frac{(a_j(\k,\nu)+1)_{\nu_j}}{\nu_j!}.
$$

\begin{defn} \label{def:Hahn}
Let $\k\in\RR^{d+1}$ with $\kappa_i>-1$ and $N\in\NN$. For $\nu\in\NN_0^d$, $|\nu|\le N$, define
the Hahn polynomials $\sH_\nu(\a; \k ,N)$ for $\a \in \ZZ_N^{d+1}$ by 
\begin{equation}\label{Hahngenfunc}
    P_{\nu,N}(y) = |y|^N \frac{P_\nu^\k \Big ( \f {y'} {|y|} \Big)}{p_\nu^\k}
   = \sum_{|\alpha| = N} \frac{N!}{\alpha!}\sH_\nu(\alpha; \kappa,N)y^\alpha,
\end{equation}
where $y = (y', y_{d+1}) \in \RR^{d+1}$. 
\end{defn}

These polynomials are given explicitly in terms of the classical Hahn polynomials  \
\begin{equation*}
  \sQ_n(x; a, b, N) := {}_3 F_2 \left( \begin{matrix} -n, n+a + b+1, -x\\
       a+1, -N \end{matrix}; 1\right), \qquad n = 0, 1, \ldots, N.    
\end{equation*}

\begin{prop} \label{prop:Hahn1}
For $x \in \ZZ_N^{d+1}$ and $\nu \in \NN_0^d$, $|\nu| \le N$, 
\begin{align}\label{eq:Hn-prod}
\sH_\nu(x;\kappa, N) =& \frac{(-1)^{|\nu|}}{(-N)_{|\nu|}} 
\prod_{j=1}^d  \frac{(\kappa_j+1)_{\nu_j}}{(a_j+1)_{\nu_j}}(-N+|\xb_{j-1}|+|\bnu^{j+1}|)_{\nu_j}  \\ 
    & \times  
    \sQ_{\nu_j}(x_j; \kappa_j, a_j, N- |\xb_{j-1}|-|\bnu^{j+1}|), \notag
\end{align}
where $a_j = a_j(\k,\nu)$ are defined in \eqref{eq:aj}. 
The polynomials in $\{\sH_\nu(x; \kappa,N): |\nu| = n\}$ form an orthogonal basis of $\CV_n^d(\sH_{\k,N})$ 
and $\sB_\nu := \la \sH_\nu(\cdot; \kappa,N), \ \sH_\nu(\cdot; \kappa,N) \ra_{\sH_{\k,N}}$ 
is given by, setting $\l_\k : = |\k|+d+1$, 
\begin{align*} 
\sB_\nu(\k, N)  :=\frac{(-1)^{|\nu|}(\l_\k)_{N+|\nu|}}
   {(-N)_{|\nu|} (\l_\k)_N (\l_{\k})_{2|\nu|}} 
   \prod_{j=1}^d \frac{(\k_j+a_j+1)_{2\nu_j}  (\k_j+1)_{\nu_j} \nu_j! }
    { (\kappa_j+a_j + 1)_{\nu_j} (a_j+1)_{\nu_j}}.
\end{align*}
\end{prop} 

The explicit formula for $\sH_\nu(\cdot; \kappa,N)$ and its norm were stated in \cite{KM} by inductive formulas. 
The generating function was later identified as the Jacobi polynomials on the simplex and the basis was given
explicitly; see \cite{IX07, X14} for further references. 

A useful observation is that $\sB_\nu(\k,N)(p_\nu^\k)^2$ and $A_\nu(\k) = \la P_\nu^\k, P_\nu^\k \ra_{W_\k}$ differ by a 
constant that depends only on $|\nu|$ (see \cite{X14}).

\begin{cor}
For $\nu \in \NN_0^d$,
\begin{equation} \label{eq:B-A}
  \sB_\nu(\k,N) = \frac{(-1)^{|\nu|}(|\k|+d+1)_{N+|\nu|}} {(-N)_{|\nu|}\, (|\k|+d+1)_N} \frac{A_{\nu}(\k)}{(p_\nu^\k)^2}. 
\end{equation}
\end{cor}

\begin{rem}
The definition of $P_\nu^\k$ in this paper differs from the one used in \cite{X14} by the constant $p_\nu^\k$; more
precisely, $P_\nu^\k(x) /p_\nu^\k$ is defined as $P_\nu^\k$ in \cite{X14}. 
\end{rem}

Let $\tau$ be a permutation in $S_{d+1}$. It is easy to see that $\{\sH_\nu (\tau\{\cdot\}, \tau \k, N): |\nu| = n\}$ is also 
an orthogonal basis of  $\CV_n^d(\sH_{\k,N})$. Hence, we can consider the connection coefficients, 
$\hs_{\mu,\nu}^\tau(\k)$, of the Hahn polynomials  
\begin{equation} \label{eq:cc-Hahn}
 \sH_\nu (\tau \a; \tau \k,N) = \sum_{|\mu| =n} \hs_{\nu,\mu}^\tau (\k) \sH_\mu (\a; \k, N).
\end{equation}
By the orthogonality, it follows immediately that 
$$
  \hs_{\nu,\mu}^\tau(\k) = \frac{1}{\sB_\mu(\k,N)} \la   \sH_\nu (\tau\{\cdot\}; \tau \k,N), \sH_\mu (\cdot; \k, N) \ra_{\sH_{\k,N}}.
$$
Let $\wh \sH_\mu (\cdot; \k,N) = [\sB_{\mu}(\k,N)]^{-1/2} \sH_\mu (\cdot; \k,N)$. Then $\{\wh \sH_\mu (\cdot; \k,N):|\mu|=n\}$ is an 
orthonormal basis of $\CV_n^d(\sH_{\k, N})$. We can write \eqref{eq:cc-Hahn} as 
$$
  \wh \sH_\nu (\tau \a; \tau \k,N) = \sum_{|\mu| = n} \wh \hs_{\nu,\mu}^\tau (\k)  \wh \sH_\mu (\a; \k,N) \quad
   \hbox{with} \quad \wh \hs_{\nu,\mu}^\tau (\k) := \sqrt{\frac{\sB_\mu(\k,N)}{\sB_\nu(\tau \k,N)}} \hs_{\nu,\mu}^\tau (\k).
$$
These connection coefficients are closely related to those of the Jacobi polynomials on the simplex, as shown in the 
following result. 

\begin{thm}
For $\nu, \mu \in \NN_0^d$ and $\tau \in S_{d+1}$, 
\begin{equation} \label{eq:con-coef1}
 \hs_{\nu,\mu}^\tau(\k) =  \frac{p_\mu^{\k}} {p_\nu^{\tau\k}}c_{\nu,\mu}^\tau (\k) \quad\hbox{and}\quad
  \wh \hs_{\nu,\mu}^\tau(\k) =  \wh c_{\nu,\mu}^{\, \tau} (\k).
\end{equation}
\end{thm}

\begin{proof}
It is known \cite[Cor. 3.5]{X14} that, 
for $\a \in \ZZ_N^{d+1}$ and $x\in \RR^d$ we have
$$
  X^\a = \frac{(\k+\one)_\a}{ (|\k|+d+1)_N} \sum_{|\mu| \le N}
              \frac{\sH_\mu(\a; \k, N)}{ \sB_\mu(\k, N)} \frac{P_\mu^\k(x)}{p_\mu^\k},
$$
where $X=(x,1-|x|)$. Since $(\tau X)^\a = X^{\tau^{-1} \a}$, using the above identity and \eqref{Hahngenfunc} with $y = \tau X$, we obtain
\begin{align*}
 \frac{P_\nu^{\tau \k}(\tau x)}{p_\nu^{\tau\k}} = & \sum_{|\a|=N} \frac{N!}{\a!}  \sH_\nu (\a;\tau\k,N) (\tau X)^\a\\
 = & \sum_{|\a|=N} \frac{N!}{\a!}  \sH_\nu (\a;\tau\k,N) 
    \frac{( \k+\one)_{\tau \a} }{ (|\k|+d+1)_N} \sum_{|\mu| \le N}
              \frac{\sH_\mu(\tau^{-1} \a; \k, N)}{ \sB_\mu(\k, N)} \frac{P_\mu^{\k}(x)}{p_\mu^{\k}}.
\end{align*}
Since the summation is over all $\a$ such that $|\a| = N$, permuting the order of the summation by $\tau$ and 
using $(\tau\a)! = \a!$, it follows that 
\begin{align*}
 \frac{P_\nu^{\tau ;\k}(x)}{p_\nu^{\tau \k}} 
    & = \frac{N!}{(|\k|+d+1)_N}   \sum_{|\a|=N} \frac{(\k+\one)_\a}{\a!} \sH_\nu (\tau \a; \tau \k,N)  
 \sum_{|\mu|\le N}\frac{\sH_\mu (\a; \k, N)} {\sB_\mu(\k,N)}  \frac{P_\mu^{\k}(x)}{p_\mu^{\k}}\\
     & =  \sum_{|\mu|\le N}  \frac{1}{ \sB_\mu(\k,N)} \la   \sH_\nu (\tau \{\cdot\}; \tau \k, N) \sH_\mu (\cdot;\k,N)
     \ra_{\sH_{\k,N}} \frac{P_\mu^{\k}(x)}{p_\mu^{\k}}\\
      & =  \sum_{|\mu|\le N}  \hs_{\nu,\mu}^\tau(\k) \frac{P_\mu^{\k}(x)}{p_\mu^{\k}}.
\end{align*}
Consequently, \eqref{eq:con-coef1} follows directly from the definition of $c_{\nu,\mu}^\tau(\k)$. Finally, using
\eqref{eq:B-A}, it is easy to verify the identity for $\wh \hs_{\nu,\mu}^\tau(\k)$. 
\end{proof}

\begin{cor}
The connection coefficients $\hs_{\nu,\mu}^\tau(\k)$ of the Hahn polynomials are independent of $N$. 
\end{cor}

In particular, for the cyclic permutation, these coefficients can be given in terms of Racah polynomials, as shown in
Theorem~\ref{thm:C(cyclic)}. 

\section{Krawtchouk polynomials of several variables}\label{se8}
\setcounter{equation}{0}

For $\rho \in \RR^d$ with $0 < \rho_i <1$, 
$1\le i \le d$ and $|\rho| < 1$, the Krawtchouk weight function of $d$ variables 
is defined by 
\begin{equation}\label{eq:weightK}
   \sK_{\rho,N}(x): =   N!\prod_{i=1}^{d} \frac{\rho_i^{x_i}}{x_i!} \frac{(1-|\rho|)^{N-|x|}}{(N-|x|)!},
   \quad x \in \NN_0^d, \quad |x| \le N,
\end{equation}
and the Krawtchouk polynomials of several variables are discrete orthogonal polynomials with respect to 
the inner product 
\begin{equation}\label{eq:ipdK}
    \la f, g \ra_{\sK_{\rho, N}}: = \sum_{x \in \NN_0^d: |x| \le N} f(x) g(x) \sK_{\rho, N}(x). 
\end{equation}

Using the notation \eqref{xsupj} for $\rho$, $\nu$ and $x$ we can define a family $\sK_\nu(\cdot;\rho,N)$ of the Krawtchouk polynomials as follows (see \cite{IX07}).

\begin{prop} \label{prop:Kraw}
Let $\rho \in \RR^d$ with $0 < \rho_i < 1$ and $|\rho| < 1$. 
For $\nu\in \NN_0^d$, $|\nu|\le N$, and $x \in \RR^d$, define
\begin{align}\label{eq:KrawK}
  \sK_\nu(x; \rho, N) := & \frac{1}{(-N)_{|\nu|}} 
    \prod_{j=1}^d 
       (-N+|\xb_{j-1}|+|\bnu^{j+1}|)_{\nu_j} \\
  &   \times 
  {}_2F_1 \left(\begin{matrix} -\nu_j,-x_j \\ -N+|\xb_{j-1}|+|\bnu^{j+1}|  \end{matrix}; \frac{1- |\brho_{j-1}|}{\rho_j} \right).\notag
\end{align}
The polynomials $\{\sK_\nu(\cdot;\rho, N): |\nu| = n\}$ form an orthogonal basis of $\CV_n^d(\sK_{\rho,N})$ 
and $\sC_\nu (\rho, N): = \la \sK_\nu(\cdot;\rho, N), \sK_\nu(\cdot;\rho, N) \ra_{\sK_{\rho, N}}$ is given by 
$$
\sC_\nu (\rho, N): =  \frac{(-1)^{|\nu|}}{(-N)_{|\nu|} } \prod_{j=1}^d \frac{\nu_j! (1-|\brho_j|)^{\nu_j+\nu_{j+1}}  } 
    {\rho_j^{\nu_j} },\quad \text{ where }\nu_{d+1}=0.
$$
\end{prop}
Note that when $d=1$, we obtain the classical Krawtchouk polynomials $\sK_n(x; p, N)$
of one variable, defined by
\begin{equation*}
  \sK_n(x; p, N) := {}_2 F_1 \left( \begin{matrix} -n, -x\\
       -N \end{matrix}; \frac{1}{p}\right), \qquad n = 0, 1, \ldots, N.    
\end{equation*}
These polynomials satisfy the following discrete orthogonality relation on $\{0,1,\ldots,N\}$ 
\begin{equation}  \label{eq:KrawKd=1}
 N! \sum_{x=0}^N \sK_n(x; p, N) \sK_m(x;p,N) \frac{p^x (1-p)^{N-x}}{x!(N-x)!} = 
     \frac{n! (N-n)! (1-p)^n}{N! p^n}. 
\end{equation}

The Krawtchouk polynomials in \eqref{eq:KrawK} can be obtained as limits of the Hahn polynomials in 
\eqref{eq:Hn-prod}. More precisely, setting $\k = t (\rho, 1-|\rho|)$, we have (see \cite{IX07})
\begin{align} \label{Hahn-Kraw}
  \lim_{t \to \infty} \sH_\nu(x; t(\rho, 1-|\rho|), N)   = (-1)^{|\nu|}\prod_{j=1}^{d}\frac{\rho_j^{\nu_j}} {(1- |\brho_j|)^{\nu_j}}\,\sK_\nu(x; \rho, N).
\end{align}
We note, however, that the polynomials in  \eqref{eq:KrawK}  differ by an unessential factor from the ones in \cite{IX07}. With the current normalization, the Krawtchouk polynomials polynomials satisfy a simple duality relation. Indeed, if we define dual indices $\nut$, variables $\xt$, and parameters $\rt$ by
\begin{equation}
\begin{alignedat}{2}
\xt_j&=\nu_{d+1-j} && \text{   for } 
                                 j=1,\dots, d,\label{eq:dualK}\\
\nut_j&=x_{d+1-j} &&\text{   for } j=1,\dots, d,  \\
\rt_j&=\frac{\rho_{d+1-j}(1-|\rho|)}{(1-|\brho_{d+1-j}|)(1-|\brho_{d-j}|)} &&\text{   for }j=1,\dots,d, 
\end{alignedat}
\end{equation}
then the following proposition holds (see \cite[Section 5.4]{GI}).
\begin{prop}\label{prop:dualityK}
The map $(x,\nu,\rho)\to(\xt,\nut,\rt)$ is an involution. Moreover, the Krawtchouk polynomials satisfy the following duality relation
\begin{equation}\label{eq:dualityK}
\sK_\nu(x;\rho, N)=\sK_{\nut}(\xt;\rt, N).
\end{equation}
\end{prop}
From the explicit formula in \eqref{eq:dualK} for the dual parameters and by using partial fractions it is easy to see that $|\rt|=|\rho|$. Combining this with the formulas for the norm and the weight of the Krawtchouk polynomials, one check that 
\begin{equation}
\sC_\nu (\rho, N)\,\sK_{\rt,N}(\xt)=(1-|\rho|)^N.
\end{equation}
As an easy consequence of the last formula and the duality in Proposition~\ref{prop:dualityK} we obtain the following corollary.
\begin{cor}\label{cor:dualorthogonalityK}
The orthonormal Krawtchouk polynomials satisfy
\begin{equation}\label{eq:dualorthogonalityK}
\sqrt{\sK_{\rho,N}(x)}\, \wh \sK_{\nu}(x; \rho, N)=\sqrt{\sK_{\rt,N}(\xt)}\, \wh \sK_{\nut}(\xt; \rt, N).
\end{equation}
\end{cor}

Let $\tau$ be a permutation in $S_{d+1}$. We denote by $\tau \rho$ the first $d$ components of $\tau (\rho, 1-|\rho|)$. 
Then $\{\sK_\nu (\tau\{\cdot\}; \tau \rho, N): |\nu| = n\}$ is also 
an orthogonal basis of  $\CV_n^d(\sK_{\rho,N})$. Hence, we can consider the connection coefficients, 
$\ks_{\mu,\nu}^\tau(\k)$, of the Krawtchouk polynomials  
\begin{equation} \label{eq:cc-Krawt}
 \sK_\nu (\tau \a; \tau \rho,N) = \sum_{|\mu| =n} \ks_{\nu,\mu}^\tau (\rho) \sK_\mu (\a; \rho, N). 
\end{equation}
By the orthogonality, it follows immediately that 
$$
  \ks_{\nu,\mu}^\tau(\rho) = \frac{1}{\sC_\mu(\rho,N)} \la   \sK_\nu (\tau\{\cdot\}; \tau \rho,N), \sK_\mu (\cdot; \rho, N) \ra_{\sK_{\rho,N}}.
$$
Let $\wh \sK_\mu (\cdot;\rho,N) = [\sC_{\mu}(\rho,N)]^{-1/2} \sK_\mu (\cdot; \rho,N)$. Then $\{\wh \sK_\mu (\cdot; \rho,N):|\mu|=n\}$ is an orthonormal basis of $\CV_n^d(\sK_{\rho, N})$ and we have
\begin{equation}\label{eq:hat-K}
  \wh \sK_\nu (\tau \a;\tau \rho,N) = \sum_{|\mu| = n} \wh \ks_{\nu,\mu}^\tau (\rho) \wh \sK_\mu (\a; \rho,N), \quad
    \wh \ks_{\nu,\mu}^\tau (\rho) := \sqrt{\frac{\sC_\mu(\rho,N)}{\sC_\nu(\tau \rho,N)}} \ks_{\nu,\mu}^\tau (\rho).
\end{equation}
These connection coefficients can be derived from those for the Hahn polynomials. Indeed, the 
limit relation \eqref{Hahn-Kraw} leads to the following lemma. 

\begin{lem}
For $\tau \in S_{d+1}$ and $|\nu| = |\mu| =n$ we have
\begin{equation} \label{eq:cc-h-k}
       \wh  \ks_{\nu,\mu}^{\tau}(\rho) =  \lim_{t \to \infty}  \wh  \hs_{\nu,\mu}^{\tau}(t(\rho, 1-|\rho|)). 
\end{equation}
In particular, the connection coefficients $\ks_{\nu,\mu}^{\tau}(\rho)$ of the Krawtchouk polynomials are 
independent of $N$. 
\end{lem}

By a limiting argument from Theorem~\ref{thm:C(cyclic)}, we can obtain a closed formula for the connection coefficients in the case of a cyclic permutation.

\begin{thm}
Let $\rho_1,\ldots, \rho_d \in (0,1)$ and $|\rho| = \rho_1+\ldots + \rho_d <1$.  
For $\nu,\mu \in \NN_0^d$ and $|\nu| = |\mu|=n$, 
\begin{align}\label{eq:CK1}
  \wh \ks_{\nu,\mu}^{(12\dots d)}(\rho) = & \ (-1)^{n+\nu_d} \sqrt{\sK_{\wh \rho,n}(\nu_1,\ldots, \nu_{d-1})}\\
       & \times   \wh \sK_{ \mu_2,\mu_3,\ldots,\mu_d} ((\nu_1,\ldots, \nu_{d-1}); \wh \rho,n), \notag
\end{align}
where $\wh \rho = (\wh \rho_1,\ldots, \wh \rho_{d-1})$ with 
\begin{equation}\label{eq:hat-rhoK}
  \wh \rho_j = \frac{\rho_1\rho_{j+1}}{(1-\rho_1)(1+\rho_1 - |\brho_{j+1}|)(1+\rho_1-|\brho_j|)}. 
\end{equation}
Furthermore, we also have
\begin{align}\label{eq:CK2}
  \wh \ks_{\nu,\mu}^{(12\dots d)}(\rho) = & \ (-1)^{n+\nu_d}\sqrt{\sK_{\rt,n}(\mu_{d},\mu_{d-1}\ldots, \mu_{2})} \\
    & \times \wh \sK_{\nu_{d-1},\nu_{d-2},\ldots,\nu_{1}} ((\mu_{d},\mu_{d-1}\ldots, \mu_{2}); \rt,n), \notag
\end{align}
where $\rt = (\rt_1,\dots, \rt_{d-1})$ with 
\begin{equation}\label{eq:hat-rhoK2}
  \rt_j = \frac{\rho_1\rho_{d+1-j}(1-|\rho|)}{(1-|\brho_{d-j}|)(1- |\brho_{d+1-j}|)(1+\rho_1-|\rho|)}. 
\end{equation}
\end{thm}

\begin{proof}
Setting $\k = t(\rho_1,\ldots, \rho_d, 1-|\rho|)$ in equation \eqref{eq:C(cyclic)} in Theorem~\ref{thm:C(cyclic)}, it is easy to check that
$$
   \b_0 = t \rho_1 \quad \hbox{and} \quad     \b_j= t(1+\rho_1 - |\brho_{d+1-j}|) +j, \quad j =1,2,\ldots d. 
$$
Taking the limit $t\to\infty$ in $\wh \sR_{\wh \mu} (\wh \nu; \beta,n)$ and ignoring positive factors independent of $\nu$ we obtain 
\begin{align*}
   &\prod_{j=1}^{d-1} (|\bmu^{d-j+2}| - |\bnu^{d-j}|)_{\mu_{d+1-j} }\\
   &\quad 
      \times{}_2F_1\left(\begin{matrix} -\mu_{d+1-j},|\bmu^{d-j+2}| - |\bnu^{d+1-j}| \\
        |\bmu^{d-j+2}| - |\bnu^{d-j}|     \end{matrix}; \frac{(1-|\brho_{d-j}|)(1+\rho_1-|\brho_{d+1-j}|)}{(1-|\brho_{d+1-j}|)(1+\rho_1-|\brho_{d-j}|)} \right) .
\end{align*}
Changing the product index from $j$ to $d-j$, the above expression becomes 
\begin{align*}
   &\prod_{j=1}^{d-1} (|\bmu^{j+2}| - |\bnu^{j}|)_{\mu_{j+1} } 
       {}_2F_1\left(\begin{matrix} -\mu_{j+1},|\bmu^{j+2}| - |\bnu^{j+1}| \\
        |\bmu^{j+2}| - |\bnu^{j}|     \end{matrix}; \frac{(1-|\brho_{j}|)(1+\rho_1-|\brho_{j+1}|)}{(1-|\brho_{j+1}|)(1+\rho_1-|\brho_{j}|)} \right) .
\end{align*}
Applying now the Pfaff identity (see \cite[Theorem 2.2.5]{AAR})
\begin{align*}
{}_2F_1\left(\begin{matrix} a, b \\
        c    \end{matrix}; x \right)=(1-x)^{-a} {}_2F_1\left(\begin{matrix} a, c-b \\
        c    \end{matrix}; \frac{x}{x-1} \right),
\end{align*}
with $a=-\mu_{j+1}$, $b=|\bmu^{j+2}| - |\bnu^{j+1}|$, $c=|\bmu^{j+2}| - |\bnu^{j}| =-n+|\bmu^{j+2}| + |\bnu_{j-1}|$ and ignoring again positive factors independent of $\nu$ we obtain 
\begin{align}
   &(-1)^{|\bmu^{2}|}\,\prod_{j=1}^{d-1} (-n+|\bmu^{j+2}| + |\bnu_{j-1}|)_{\mu_{j+1} }\label{eq:Kr2}\\
   &\quad 
      \times {}_2F_1\left(\begin{matrix} -\mu_{j+1},- \nu_{j} \\
        -n+|\bmu^{j+2}| + |\bnu_{j-1}| \end{matrix}; \frac{(1-|\brho_{j}|)(1+\rho_1-|\brho_{j+1}|)}{\rho_1\rho_{j+1}}\right) .\notag
\end{align}
The parameters $\wh \rho_j$ defined in \eqref{eq:hat-rhoK} satisfy 
$$ \frac{1- |\wh \brho_{j-1}|}{\wh \rho_j}=\frac{(1-|\brho_{j}|)(1+\rho_1-|\brho_{j+1}|)}{\rho_1\rho_{j+1}},$$
and comparing the last two formulas with  \eqref{eq:KrawK} shows that, up to  a positive factor independent of $\nu$, equation \eqref{eq:Kr2} coincides with the orthonormal polynomial $ \wh \sK_{ \mu_2,\mu_3,\ldots,\mu_d} ((\nu_1,\ldots, \nu_{d-1}); \wh \rho,n)$. 
Furthermore, taking the limit in the weight function 
$w_\sR(\wh \nu, \beta,n)$ in \eqref{eq:C(cyclic)}  and using the formula \eqref{eq:d-Racah-weight}, we obtain
$$
  \frac{(1-|\rho|)^{\nu_d} (1+\rho_1-|\rho|)^{\nu_d}} {\nu_d! \rho_1^{\nu_d}} 
     \prod_{j=1}^{d-1}  \frac{\rho_{j+1}^{\nu_j}} {\nu_j!} \frac{(1+\rho_1-|\brho_j|)^{ |\bnu^j|+ |\bnu^{j+1}|}}
         {(1+\rho_1-|\brho_{j+1}|)^{ |\bnu^j|+ |\bnu^{j+1}|}},
$$
which can be simplified to 
$$
     \prod_{j=1}^{d-1}  \frac{\rho_{j+1}^{\nu_j}} {\nu_j!}  \frac{(1-|\rho|)^{\nu_d} } {\nu_d! } 
     \frac{1} {\rho_1^{\nu_d} (1+\rho_1-|\brho_{j+1}|)^{\nu_j+\nu_{j+1}}}. 
$$
Up to a factor independent of $\nu$, this is precisely the weight function $\sK_{\wh \rho, n}(\nu_1,\ldots,\nu_{d-1})$
as can be directly verified using \eqref{eq:hat-rhoK} and 
$$
  1- |\wh \rho| =  \frac{1- |\rho|}{(1-\rho_1)(1+\rho_1 -|\brho|)}.
$$
The latter can be directly verified by writing $\wh \rho_j$ in partial fraction. Positive multiplicative factors independent of the variable
$\nu$ are not accounted for in the above computations, but they are unessential since an analogue of \eqref{eq:discreteOP} 
for $\wh \ks_{\nu,\mu}^\tau$ completes the proof of \eqref{eq:CK1}. Applying now Corollary~\ref{cor:dualorthogonalityK} we obtain equation \eqref{eq:CK2}.
\end{proof}

\section{Orthogonal polynomials on the unit ball and the unit sphere}\label{se9}
\setcounter{equation}{0}

For the unit ball $\BB^d = \{x: \|x\|\le 1\}$, where $\|x\|$ denotes the Euclidean norm of $x$ in $\RR^d$, the classical weight function is given by
$$
  W_\mu^\BB(x) = (1-\|x\|^2)^{\mu}, \quad \mu > -1, \quad x \in \BB^d. 
$$
A more general family of weight functions on the ball is defined by
$$
  W_\k^\BB(x) = \prod_{i=1}^d |x_i|^{2 \k_i+1} (1-\|x\|^2)^{\k_{d+1}}, \quad \k_i > -1, \quad x \in \BB^d. 
$$
One orthogonal basis for $\CV_n^d(W_\k^\BB)$ is given by $\{P_\a^n(W_\k^\BB; \cdot):|\a| =n\}$ with
\begin{equation} \label{eq:basisBall}
 P_\a^n(W_{\k}^\BB;x) := \prod_{j=1}^d (1-\|\xb_{j-1}\|^2)^{\a_j/2} C_{\a_j}^{(\l_j+\f12,\k_j+\f12)}\left( \frac{x_j}{\sqrt{1-\|\xb_{j-1}\|^2}} \right),
\end{equation}
where $\l_j = \l_j(\k,\a):= |\bal^{j+1}|+|\bka^{j+1}|+d-j$ and $C_n^{(\l,\mu)}$ denotes the generalized Gegenbauer polynomials that are orthogonal with respect to 
$|x|^{2\mu}(1-x^2)^{\l-\f12}$ on $[-1,1]$; see \cite[p.\ 266]{DX}. It should be noted that the parameters in our weight function differ from those in 
\cite{DX} by an additive $1/2$. 

The orthogonal polynomials on the unit ball are known to be related to orthogonal polynomials on the simplex. In fact, if 
$\{P_\nu^\k: |\nu|=n\}$ is a basis of $\CV_n^d(W_{\k})$ on the simplex, it is known \cite[Section 4.4]{DX} that 
$\{P_\nu^\k(x_1^2,\ldots,x_d^2): |\nu| =n\}$ is a basis for the subspace of $\CV_n^d(W_\k^\BB)$ that consists of 
orthogonal polynomials that are invariant under sign changes. The proposition below completes the relation in the other direction.

\begin{defn}\label{defn:Q-ball}
Let $\{P_\nu^\k: |\nu|=n\}$ be a basis of $\CV_n^d(W_{\k})$ on the simplex.
Let $\ve \in \{0,1\}^d$ and $\ve' = (\ve,0)$. Assume $(n- |\ve|)/2 \in \NN_0$. For $\nu \in \NN_0^d$ with $|\nu| =(n- |\ve|)/2$, define 
$$
    Q_{\nu,\ve}^\k (x) := x^\ve P_\nu^{\k+\ve'}(x_1^2,\dots,x_d^2).
$$ 
\end{defn}
It is evident that $Q_{\nu,\ve}^\k$ is a polynomial of degree exactly $n$. To simplify the notation, we write below $\ve$ instead of $\ve'$, i.e. for $\ve\in\{0,1\}^d$ when we need to work in $\RR^{d+1}$ we think of $\ve$ as $(\ve,0)$.

\begin{prop} \label{prop:ball}
If $\{P_\nu^\k: |\nu|=n\}$ is an orthogonal basis of $\CV_n^d(W_{\k})$, then the set 
$X_n:=\{Q_{\nu,\ve}^\k: |\nu| = \frac{n- |\ve|}{2} \in \NN_0, \ve \in \{0,1\}^d\}$ is an orthogonal basis
for $\CV_n^d(W_\k^\BB)$. 
\end{prop}

\begin{proof}
For $n -|\ve|$ being an even integer and $\b\in \NN_0^d$ with $|\b| < n$, the integral
\begin{equation*} 
  I(\beta):= \int_{\BB^d} Q_{\nu,\ve}^\k(x) x^\beta W_\k^\BB(x) dx =  \int_{\BB^d} x^{\ve} P_\nu^{\k+\ve}(x_1^2,\ldots,x_d^2) x^\beta W_\k^\BB(x) dx 
\end{equation*}
is equal to zero if $\b-\ve $ contains an odd component, as can be seen by changing variable $x_i \mapsto - x_i$. 
In the remaining case, we can write $\beta = \ve + 2 \g$ with $|\g| < \frac{n-|\ve|}{2}$, so that 
$$
  I(\ve+ 2\g) = \int_{\BB^d} x^{2\g} P_\nu^{\k+\ve}(x_1^2,\ldots,x_d^2) x^{2\ve} W_\k^\BB(x) dx. 
$$
Since $x^{2\ve} W_\k^\BB(x) = W_{\k+\ve}(x_1^2,\ldots,x_d^2) {x_1\cdots x_d}$, it follows from the identity 
$$
  \int_{\BB^d}f(x_1^2,\ldots, x_d^2) dx = \int_{T^d} f(x_1,\ldots,x_d) \frac{dx}{\sqrt{x_1\cdots x_d}} 
$$
that $I(\ve + 2 \g) =0$ by the orthogonality of $P_\nu^{\k+\ve}$ since $|\g| < |\nu|$. The mutual orthogonality of 
$Q_{\nu,\ve}^{\k}$ in $X_n$ follows from the parity of $x^\ve$. Finally, it is easy to check that the cardinality of $X_n$
is equal to $\dim \CV_n^d$, so that $X_n$ is a basis of $\CV_n^d(W_\k^\BB)$. 
\end{proof}

If $P_\nu^\k(x)$ are the Jacobi polynomials on the simplex, defined in \eqref{eq:Pnu} of Section \ref{se2}, the basis $Q_{\nu,\ve}^\k$ in the proposition above is exactly the basis \eqref{eq:basisBall} on the unit ball.  

\begin{prop}\label{prop:ball2}
Let $P_\nu^\k$ be the Jacobi polynomials on the simplex. Then, up to unessential constant factors, the set $X_n$ in Proposition~\ref{prop:ball} coincides with the basis $\{P^n_\a(W_\k^\BB:\cdot); |\a| = n, \a \in \NN_0^d\}$. 
\end{prop}

\begin{proof}
The generalized Gegenbauer polynomials $C_n^{(\l,\mu)}$ are related to the Jacobi polynomials by (\cite[p.\ 25-26]{DX})
\begin{align*}
C_{2m}^{(\l,\mu)}(t) & = \frac{(\l+\mu)_m}{(\mu+\f12)_m} P_m^{(\l-\f12,\mu-\f12)}(2t^2-1) \\
C_{2m+1}^{(\l,\mu)}(t) & = \frac{(\l+\mu)_{m+1}}{(\mu+\f12)_{m+1}} t  P_m^{(\l-\f12,\mu+\f12)}(2t^2-1).
\end{align*}
Using the above two identities, \eqref{eq:Pnu}, \eqref{eq:basisBall} and that $a_j(\k+\ve,\nu) = 
\l_j(\k, 2\nu+\ve)$ for $a_j$ defined in \eqref{eq:aj}, one can verify that 
\begin{equation}\label{eq:ball2}
  x^{\ve} P_\nu^{\k+\ve}(x_1^2,\ldots,x_d^2) \equiv P^n_{2\nu+ \ve}(W_\k^\BB; x)
\end{equation}
where the symbol $\equiv$ means that the equation holds up to a multiple constant.
\end{proof}
 
Let $\tau$ be a permutation in $S_d$. Let $\tau \k = (\tau (\k_1,\ldots,\k_d),\k_{d+1})$. 
For $x \in \RR^d$, the weight function $W_\k^\BB$ satisfies $W_{(\tau \k, \k_{d+1})}^\BB(\tau y) = W^{\BB}_\k(y)$. It is easy to see that the set 
$\{P_\a^n(W_{\tau \k}^{\BB};\tau x): |\a|=n\}$ is also an orthogonal basis of $\CV_n^d(W_\k^{\BB})$. Consequently, 
we can consider the connection coefficients $b_{\a,\b}^{\tau}(\k)$ defined by 
$$
    P_\a^n(W_{\tau \k}^{\BB};\tau x) = \sum_{|\b| =n} b_{\a,\b}^{\tau}(\k) P_\b^n(W_\k^{\BB};x).
$$
From the above discussion, it is evident that $b_{\nu,\mu}^\tau(\k)$ can be derived from $c_{\nu,\mu}^\tau(\k)$. 

\begin{thm} \label{prop:ball-3}
Let $\ve \in \{0,1\}^d$. Then, for $v\in \NN_0^d$, the connection coefficients $b_{2\nu+\ve,\b}^\tau(\k) = 0$ if $\b \ne 2 \mu +\ve$ for 
some $\mu \in \NN_0^d$ and  
\begin{equation}\label{eq:ball3}
      \wh b_{2\nu+\ve,2\mu+\ve}^{\,\tau}(\k) = \wh c_{\nu,\mu}^{\,\tau}(\k+\ve).
\end{equation}
\end{thm}

\begin{proof}
That $b_{2\nu+\ve,\b}^\tau(\k) = 0$ when $\beta$ does not have the same parity as $2 \nu + \ve$ follows 
directly from Proposition \ref{prop:ball2}, using the connection coefficients of $c_{\nu,\mu}^{\,\tau}(\k)$. The
case \eqref{eq:ball3} follows as a consequence of \eqref{eq:ball2} since we can ignore the irrelevant constant factors when we consider normalized connection coefficients. 
\end{proof}
 
In the above consideration we assume $\tau \in S_d$ instead of $S_{d+1}$ to exclude the permutation in $S_{d+1}$
acting on $\k_{d+1}$, because such a permutation will have to act on $x_{d+1} := \sqrt{1-\|x\|^2}$ as well in order to 
keep the weight function invariant. However, $x_{d+1}$ is not a polynomial. To include permutations that
act on $\k_{d+1}$, we need to go back to the orthogonal polynomials on the simplex. Let us illustrate the situation
in the particular case $d=2$ for the classical weight function
$$
   W_\mu^\BB(x) = (1-x_1^2-x_2^2)^\mu, \quad \mu > -1
$$
on the disk $\BB^2$. There are two orthogonal bases that are well known \cite[Sect. 2.3]{DX}. 
The first basis of $\CV_n^2(W_\mu^\BB)$ is \eqref{eq:basisBall}, in Cartesian coordinates, which we reindex by 
\begin{equation} \label{eq:basis1_ball}
   P^\BB_{j,n}(x_1,x_2) =  C_{n-j}^{j+\mu+1}(x_1) 
      (1-x_1^2)^{\frac{j}{2}}  C_j^{\mu+\f12} \biggl(\frac{x_2}{\sqrt{1-x_1^2}} \biggl),  \quad 0 \le j \le n,
\end{equation}
whereas the second one is given, in polar coordinates $(x_1,x_2) = (r\cos \t, r\sin \t)$, by 
\begin{align} \label{basis2_ball}
\begin{split} 
 Q_{j,1}^n (x_1,x_2) & =   P_{j}^{(\mu,n-2j)}(2r^2 -1)r^{n-2j}\cos (n-2j)\t, \quad 0 \le j \le \tfrac{n}2,\\
 Q_{j,2}^n (x_1,x_2) & =  P_{j}^{(\mu,n-2j)}(2r^2 -1)r^{n-2j}\sin (n-2j) \t, \quad 0 \le j < \tfrac{n}2.
\end{split}
\end{align}

Recall the orthogonal polynomials $P_{n-j,j}^\k$, defined in \eqref{eq:Pjn}, with respect to the weight function 
$W_\k(x_1,x_2) = x_1^{\k_1} x_2^{\k_2} (1-x_1-x_2)^{\k_3}$ on the triangle $T^2$. The first basis 
$\{P_{j,n}^\BB: 0\le j \le n\}$, as shown by Proposition \ref{prop:ball2}, arises from 
$P_{n-j,j}^\k$ with $\k = (\pm \f12, \pm \f12, \mu)$. 

\begin{prop} \label{prop:2basisBall}
The basis \eqref{basis2_ball} agrees, up to a multiple constant, with the basis in Proposition \ref{prop:ball2} that arises
from $\{P_{n-j,j}^{(13);\k}: 0 \le j \le n\}$ with $\k = (-\f12,-\f12,\mu)$.
\end{prop}

\begin{proof}
For $\k = (-\f12,-\f12,\mu)$, we obtain by \eqref{eq:Pjn(13)} that, for $0\le \ell \le n$,  
\begin{align*}
  P_{n-\ell,\ell}^{(13) ;\k} (x_1^2,x_2^2) = (-1)^{n-\ell} P_{n-\ell}^{(\mu,2\ell)} (2x_1^2+2 x_2^2-1) (x_1^2+x_2^2)^\ell
       P_\ell^{(-\f12,-\f12)}\left(\frac{2x_2^2}{x_1^2+x_2^2} -1\right).
\end{align*}
Since $P_\ell^{(-\f12, -\f12)}(\cos \t)$ is, up to a multiple constant, the Chebyshev polynomial $T_\ell(\cos \t) = \cos \ell \t$,
and $\frac{2x_2^2}{x_1^2+x_2^2} -1 = 2 \sin^2 \t -1 = - \cos 2 \t$, written in polar coordinate gives
$$
   P_{n-\ell,\ell}^{(13) ;\k} (x_1^2,x_2^2)   \equiv P_{n-\ell}^{(\mu,2\ell)} (2r^2-1)  r^{2\ell}
            \cos 2 \ell \t,  
$$
where we use $\equiv$ to denote that the identity holds up to a constant, which is evidently $Q_{n-\ell,1}^{2n}(x_1,x_2)$. 
Similarly, for $\k = (\f12,\f12,\mu)$, we obtain by \eqref{eq:Pjn(13)} that, for $0\le \ell \le n-1$,  
\begin{align*}
  x_1x_2 P_{n-\ell-1,\ell}^{(13) ;\k} (x_1^2,x_2^2) \equiv P_{n-\ell-1}^{(\mu,2\ell+2)} (2r^2 -1) r^{2 \ell+2}
         \sin 2\t  P_\ell^{(\f12,\f12)}(\cos 2 \t).
\end{align*}
Since $P_\ell^{(\f12,\f12)}(\cos 2 \t)  \equiv \frac{\sin (2\ell+2)\t}{\sin 2\t}$, it follows that the above polynomial is, 
up to a multiple constant, $Q_{n-\ell-1,2}^{2n}(x_1,x_2)$. This completes the proof for the basis of even degree. 
Using the identities
$$
   P_\ell^{(\f12, - \f12)}(\cos 2 \t) \equiv \frac{\sin (2\ell+1)\t}{\sin \t} \quad \hbox{and} \quad
    P_\ell^{(- \f12, \f12)}(\cos 2 \t) \equiv \frac{\cos (2\ell+1)\t}{\cos \t},
$$
the proof for the basis of odd degree follows similarly. 
\end{proof}

For the two bases in \eqref{eq:basis1_ball} and \eqref{basis2_ball}, we can define the connection coefficients by
$$
  \wh Q_{n-j,i}^n(x_1,x_2) = \sum_{m=0}^n \wh b_{j,m}^{n,i} \wh P_{m,n}^\BB(x_1,x_2), \quad i =1,2, 
$$
where $0 \le j,m \le n$ and the hat symbol again denotes orthonormality. These coefficients can be deduced from 
the connection coefficients in Section \ref{se4},  as shown by Theorem \ref{prop:ball-3}. For example,  since 
$\wh P_{2m,2n}^\BB (x_1,x_2) = \wh P_{n-m,m}^\k(x_1^2,x_2^2)$ and $\wh Q_{n-\ell,1}^{2n} =  \wh P_{n-\ell,\ell}^{(13);\k}(x_1^2,x_2^2)$, 
where $\k = (-\f12,-\f12,\mu)$, it follows by the definition of $\wh c_{\ell,m}^\tau(\k)$,  
$$
  \wh Q_{n-\ell,1}^{2n} = \wh P_{n-\ell,\ell}^{(13);\k} = \sum_{m=0}^n \wh c_{\ell,m}^{(13)}(\k) \wh P_{n-m,m}^\k
       =  \sum_{m=0}^n \wh c_{\ell,m}^{(13)}(\k) \wh P_{2m,2n}^\BB.
$$
Hence, $\wh b_{\ell,2m+1}^{2n,1} = 0$ and $\wh b_{\ell,2m}^{2n,1} = \wh c_{\ell,m}^{(13)}(\k)$. 
In particular, $\wh  b_{\ell,2m}^{2n,1}$ can be written in terms of a Racah polynomial. The other cases can be 
worked out similarly. 

A similar result holds in higher dimension for the weight function $W_\mu$. The definition of the corresponding basis
in the polar coordinates will involve spherical harmonics defined on the unit sphere $\SS^d$ of $\RR^{d+1}$, which
are homogeneous polynomials orthogonal over $\SS^d$ with respect to the surface measure $d\s$. More generally,
we can consider the weight function 
$$
  w_\k(x) = \prod_{i=1}^{d+1} |x_i|^{2\k_i + 1}, \quad \k_i > -1, \quad x \in \SS^d.
$$
Evidently $w_\k(x) d\s$ becomes $d\s$ if $\k_1=\cdots = \k_{d+1} = -\f12$. Let $\CH_n^{d+1}(w_\k)$ be the space of 
homogeneous polynomials of degree $n$ orthogonal with respect to $w_\k$ on $\SS^d$. It is known that an orthogonal basis for 
$\CH_n^{d+1}(w_\k)$ can be derived from orthogonal polynomials on the unit ball. More precisely, for $y\in \RR^{d+1}$, 
write $y = r (x,x_{d+1})$; then $x \in \BB^d$ and $(x,x_{d+1}) \in \SS^d$. Let $\varphi(x) = \sqrt{1-\|x\|^2}$. 
Define, for $\nu \in \NN_0^d$, homogeneous polynomials 
\begin{align} \label{eq:YnPmu}
\begin{split}
 Y_\nu^{(1)}(y):= & r^n P^{n}_\nu (\varphi^{-1} W_\k^{\BB}; x), \qquad |\nu|= n, \\
 Y_\nu^{(2)}(y):= & r^n x_{d+1} P^{n-1}_\nu (\varphi W_\k^{\BB};x), \qquad |\nu|= n-1.
\end{split}
\end{align}
Then the set $\{Y_\nu^{(1)}: |\nu| =n\}\cup \{Y_\nu^{(2)}:|\nu| = n-1\}$, restricted to $\SS^d$, is an orthogonal 
basis for $\CH_n^{d+1}(w_\k)$ (see, for example, \cite[Section 4.2]{DX}). 

Comparing the above with Definition \ref{defn:Q-ball},  we can add another parity at the last variable and establish the following result. 

\begin{defn}
Let $\{P_\nu^\k: |\nu| =n\}$ be an orthogonal basis of $\CV_n^d(W_\k)$ on the simplex. 
Let $\ve \in \{0,1\}^{d+1}$ and $y = r(x,x_{d+1})$. Assume $(n- |\ve|)/2 \in \NN_0$. For $\nu \in \NN_0^d$ with $|\nu| =(n- |\ve|)/2$, 
define 
$$
    Q_{\nu,\ve}^\k (y) := r^{n-|\ve|} y^\ve P_\nu^{\k+\ve}(x_1^2,\ldots,x_d^2), \qquad  (x_1,\ldots,x_{d+1}) \in \SS^d.
$$ 
\end{defn}

It is not difficult to see that $Q_{\nu,\ve}^\k$ is a homogeneous polynomial of degree $n$ in $y$. Comparing with Definition
\ref{defn:Q-ball}, we see that the following proposition holds as a consequence of Proposition \ref{prop:ball} and \eqref{eq:YnPmu}. 

\begin{prop} 
The polynomial $Q_{\nu,\ve}^\k$ is a homogeneous polynomial in $y$ and the collection of all such polynomials restricted to the sphere, 
$\{Q_{\nu,\ve}^\k: \ve \in \{0,1\}^{d+1}, |\nu|=(n- |\ve|)/2 \in \NN_0\}$, is an orthogonal basis of $\CH_n^{d+1}(w_\k)$.  
\end{prop}

Note that, restricted to the sphere, $x_{d+1} = \sqrt{1-\|x\|^2}$ is a polynomial of degree $1$. When $P_\nu^\k$ are the Jacobi polynomials on the simplex and $(x,x_{d+1})$ is written in the standard spherical coordinates,
the basis in the above proposition is precisely the one given in \cite[Theorem 7.5.2]{DX}. In particular, this gives the standard basis for the spherical harmonics when $\k_1 =\ldots= \k_{d+1} = -\f12$, see, \cite[Theorem 4.1.4]{DX}.  

As an example, let us state explicitly this basis for $\CH_n^3$ of spherical harmonics of three variables in terms of the orthogonal 
polynomials on the simplex. If $Y$ is a homogeneous polynomial, then $Y(x) = r^n Y(x')$ for $x = r x'$ with $r > 0$ and $\|x' \|=1$. 
Hence, we only need to give the basis in terms of $x'$. 

\begin{exam} 
Orthogonal basis for $\CH_{n}^3$ of the spherical harmonics on $\SS^2$. Let $\{P_{n-k,k}^{\left(\k_1,\k_2,\k_3\right)}(x,y):0\le k \le n\}$ 
be the Jacobi polynomials defined in \eqref{eq:Pjn}. Then, for $(x_1, x_2,x_3) \in \SS^2$, 
\begin{align*}
  & \{P_{n-j,j}^{(-\f12,-\f12,-\f12)}(x_1^2,x_2^2), 0 \le j \le n\} \cup  \{x_1 x_2 P_{n-1-j,j}^{(\f12,\f12,-\f12)}(x_1^2,x_2^2), 0 \le j \le n-1\} \\
   &\, \cup \{x_1x_3 P_{n-1-j,j}^{(\f12,-\f12,\f12)}(x_1^2,x_2^2), 0 \le j \le n-1\} \cup \{x_2x_3 P_{n-1-j,j}^{(-\f12,\f12,\f12)}(x_1^2,x_2^2), 0 \le j \le n-1\}
\end{align*}
is an orthogonal basis for $\CH_{2n}^3$; and 
\begin{align*}
  & \{x_1 P_{n-j,j}^{(\f12,-\f12,-\f12)}(x_1^2,x_2^2), 0 \le j \le n\} \cup  \{x_2 P_{n-j,j}^{(-\f12,\f12,-\f12)}(x_1^2,x_2^2), 0 \le j \le n\} \\
   &\, \cup \{x_3 P_{n-j,j}^{(-\f12,-\f12,\f12)}(x_1^2,x_2^2), 0 \le j \le n\} \cup \{x_1x_2x_3 P_{n-1-j,j}^{(\f12,\f12,\f12)}(x_1^2,x_2^2), 0 \le j \le n-1\}
\end{align*}
is an orthogonal basis for the space $\CH_{2n+1}^3$.
\end{exam}

To see directly that this basis is indeed the classical basis of spherical harmonics on $\SS^2$, we need to use the standard spherical 
coordinates and carry out computations as in the proof of Proposition \ref{prop:2basisBall}. We leave it to the interested readers. 

The restriction of $\{Q_{\nu,\ve}^{\tau \k}(\tau (x,x_{d+1})): \ve \in \{0,1\}^{d+1}, (n- |\ve|)/2 \in \NN_0\}$ is evidently another 
orthogonal basis of $\CH_n^d(w_\k)$. We can define the connection coefficients between this basis and the basis $\{Q_{\nu,\ve}^\k\}$. 
There is no more mystery or work needed to be done in this case. These connection coefficients can directly be computed from those for the Jacobi polynomials on the simplex. 

\section*{Acknowledgments}
The authors thank two anonymous referees for their thoughtful comments that 
helped improve the presentation of the paper.

\end{document}